\documentclass[11pt, letterpaper]{amsart}

\usepackage{color}
\usepackage{verbatim}
\usepackage{comment}

\usepackage[utf8x]{inputenc}
\usepackage{t1enc}
\usepackage[english]{babel}
\usepackage{pgfplots}
\pgfplotsset{compat=1.15}
\usepackage{mathrsfs}
\usetikzlibrary{arrows}

\usepackage[noadjust]{cite}

\usepackage{amsmath,amssymb,amsthm}
\usepackage{mathrsfs,esint}
\usepackage{graphicx,wrapfig}
\usepackage[margin=10pt]{caption}
\newcommand{\R}{\mathbb{R}}

\newcommand{\N}{\mathbb{N}}

\newcommand{\Ee}{\mathcal{E}}

\newcommand{\D}{\mathbb{D}}

\newcommand{\Ha}{\mathcal{H}}
\newcommand{\K}{\mathcal{K}}
\newcommand{\Leb}{\mathcal{L}}

\newcommand{\eps}{\varepsilon}

\newcommand{\loc}{\text{loc}}

\newcommand{\phii}{\varphi}

\newcommand{\bmat}{\begin{bmatrix}}
\newcommand{\emat}{\end{bmatrix}}
\newcommand{\til}{\tilde}
\newcommand{\wtil}{\widetilde}

\providecommand*{\vint}[1]{\mathchoice
          {\mathop{\vrule width 5pt height 3 pt depth -2.5pt
                  \kern -9pt \kern 1pt\intop}\nolimits_{\kern -5pt{#1}}}
          {\mathop{\vrule width 5pt height 3 pt depth -2.6pt
                  \kern -6pt \intop}\nolimits_{\kern -3pt{#1}}}
          {\mathop{\vrule width 5pt height 3 pt depth -2.6pt
                  \kern -6pt \intop}\nolimits_{\kern -3pt{#1}}}
          {\mathop{\vrule width 5pt height 3 pt depth -2.6pt
                  \kern -6pt \intop}\nolimits_{\kern -3pt{#1}}}}

\DeclareMathOperator{\Ext}{Ext}

\DeclareMathOperator{\dist}{dist}
\DeclareMathOperator{\diam}{diam}
\DeclareMathOperator{\rad}{rad}

\numberwithin{equation}{section}
\theoremstyle{plain}
\newtheorem{thm}[equation]{Theorem}

\newtheorem{prop}[equation]{Proposition}

\newtheorem{lem}[equation]{Lemma}

\theoremstyle{definition}

\newtheorem{defn}[equation]{Definition}

\newtheorem{remark}[equation]{Remark}

\newtheorem{example}[equation]{Example}

\begin{document}

\title[Obstacle problems for least gradients]{Obstacle problems with double boundary condition for least gradient functions in metric measure spaces }

\author{Josh Kline} 

\date{October 19, 2022}

\maketitle

\begin{abstract}
In the setting of a metric space equipped with a doubling measure supporting a $(1,1)$-Poincar\'e inequality, we study the problem of minimizing the BV-energy in a bounded domain $\Omega$ of functions bounded between two obstacle functions inside $\Omega$, and whose trace lies between two prescribed functions on the boundary of $\Omega.$  If the class of candidate functions is nonempty, we show that solutions exist for continuous obstacles and continuous boundary data when $\Omega$ is a uniform domain whose boundary is of positive mean curvature in the sense of Lahti, Mal\'y, Shanmugalingam, and Speight (2019).  While such solutions are not unique in general, we show the existence of unique minimal solutions.  Our existence results generalize those of Ziemer and Zumbrun (1999), who studied this problem in the Euclidean setting with a single obstacle and single boundary condition.   
\end{abstract}

\section{Introduction}

Given some function $f$ on the boundary of a domain $\Omega$ and a function $\psi$ on $\overline\Omega$, the obstacle problem for least gradient functions is the problem of minimizing the BV-energy in $\Omega$ over all functions $u\in BV(\Omega)$ whose trace agrees with $f$ almost everywhere on $\partial\Omega$ \emph{and} such that $u\ge\psi$ almost everwhere in $\Omega.$  In the Euclidean setting, this problem was first studied by Ziemer and Zumbrun in \cite{ZZ}, where they showed that a continuous solution exists for $\psi\in C(\overline\Omega)$ and $f\in C(\partial\Omega)$ such that $f\ge\psi$ on $\partial\Omega,$ provided the boundary of $\Omega$ has nonnegative mean curvature and is not locally area-minimizing.  Their work generalized some of the earlier results of Sternberg, Williams, and Ziemer \cite{SWZ}, who introduced and studied the Dirichlet problem for least gradient functions (without obstacle) in the Euclidean setting.  Recently existence, uniqueness, and regularity of an anisotropic formulation of the obstacle problem for least gradients in the Euclidean setting was studied in \cite{FM}.  For the relationship between the obstacle problem for least gradients and dual maximization problems in the Euclidean setting, see \cite{SS}.  

In \cite{L}, existence and regularity of solutions to the obstacle problem for least gradients were studied in the setting of a metric space equipped with a doubling measure and supporting a Poincar\'e inequality.  Here solutions were not required to attain the boundary condition in the sense of traces; competing functions need only satisfy the obstacle condition inside the domain and satisfy the boundary condition outside the domain.

In this paper, we continue the study of obstacle problems for least gradient functions in the metric setting.  In contrast to \cite{L} however, we insist that solutions address the boundary condition in the sense of traces, following \cite{ZZ} in the Euclidean setting.  Furthermore, we include a second obstacle function and consider a \emph{double} boundary condition.  That is, for $\psi_1,\psi_2:\overline\Omega\to\overline\R$ and $f,g\in L^1(\partial\Omega),$ we consider the problem
\begin{equation*}
\min\{\|Du\|(\Omega):u\in BV(\Omega),\,\psi_1\le u\le\psi_2,\, f\le Tu\le g\},
\end{equation*}
where the inequalities are in the almost everywhere sense.  By defining $\K_{\psi_1,\psi_2,f,g}(\Omega)$ to be the class of BV-functions in $\Omega$ satisfying the obstacle and boundary conditions, we refer to solutions to this problem as {\it strong solutions} to the $\K_{\psi_1,\psi_2,f,g}$-obstacle problem (see Section~2.3 below for the precise definitions).  We adopt the following standing assumptions: 
\begin{itemize}
\item $(X,d,\mu)$ is a complete metric measure space supporting a $(1,1)$-Poincar\'e inequality, with $\mu$ a doubling Borel regular measure,\vskip.2cm
\item $\Omega\subset X$ is a bounded domain with $\mu(X\setminus\Omega)>0,$\vskip.2cm
\item $\Ha(\partial\Omega)<\infty$, $\Ha|_{\partial\Omega}$ is doubling, and $\Ha|_{\partial\Omega}$ is lower codimension 1 Ahlfors regular, see \eqref{eq:LowerAR},
\item $\Ha(\{z\})=0$ for all $z\in\partial\Omega.$
\vskip.2cm
\end{itemize}
where $\Ha$ is the codimension 1 Hausdorff measure on $\partial\Omega$, see \eqref{eq:HaMeasure}.  The assumption that singletons on the boundary of the domain are $\Ha$-negligible is necessary to obtain Lemma~\ref{lem:SolnSetBound}, see Example~\ref{ex:Singletons}. Our main result is the following: 

\begin{thm}\label{thm:ExistStrongSoln}
Let $\Omega$ be a uniform domain with boundary of positive mean curvature as in Definition~\ref{defn:PMC}.  Let $f,g\in C(\partial\Omega)$ and $\psi_1,\psi_2\in C(\overline\Omega)$ be such that $\K_{\psi_1,\psi_2,f,g}(\Omega)\ne\varnothing.$ Then there exists a strong solution to the $\K_{\psi_1,\psi_2,f,g}$-obstacle problem. 
\end{thm}
By $\psi_1,\psi_2\in C(\overline\Omega),$ we mean that these functions are extended real valued functions, and continuous with respect to the standard topology on the extended real line.  As such we can consider continuous obstacle functions $\psi_1\equiv-\infty$ and $\psi_2\equiv\infty.$  However throughout this paper, we do insist that the boundary functions $f$ and $g$ are real-valued, so as to utilize certain extension results from \cite{LMSS}.

Strong solutions to this problem may fail to be unique (see \cite{LMSS} and the discussion below), and so a comparison theorem for strong solutions will not hold in general.  However, in Proposition~\ref{prop:DObsExistMin} and Remark~\ref{remark:MaxSoln}, we show that unique \emph{minimal} and \emph{maximal} strong solutions exist for continuous obstacle and boundary functions, and we obtain the following comparison-type result:

\begin{thm}\label{cor:DObsCompThm}
Let $f_1,f_2,g_1,g_2\in C(\partial\Omega)$ such that $f_1\le f_2$ and $g_1\le g_2$ $\Ha$-a.e.\ on $\partial\Omega.$  Let $\psi_1,\psi_2,\phii_1,\phii_2\in C(\overline\Omega)$ be such that $\psi_1\le\psi_2$ and $\phii_1\le\phii_2$ $\mu$-a.e.\ in $\Omega.$  Suppose that 
$$\K_1(\Omega):=\K_{\psi_1,\phii_1,f_1,g_1}(\Omega)\ne\varnothing\ne\K_{\psi_2,\phii_2,f_2,g_2}(\Omega)=:\K_2(\Omega).$$
If $u_1$ is the minimal strong solutions to the $\K_1$-obstacle problem and $u_2$ is a strong solution (not necessarily minimal) to the $\K_2$-obstacle problem, then $u_1\le u_2$ $\mu$-a.e.\ in $\Omega.$  Likewise, if $v_1$ is a strong solution to the $\K_1$-obstacle problem and $v_2$ is the maximal strong solution to the $\K_2$-obstacle problem, then $v_1\le v_2$ $\mu$-a.e.\ in $\Omega.$      
\end{thm}

Problems involving related double boundary conditions have been studied in a variety of different contexts.  In stochastic analysis in particular, double-boundary (non-)crossing problems, where one tries to determine the probability that a stochastic process remains between two prescribed boundaries, have been studied extensively and have many statistical applications.  For a sampling, see \cite{BGNR,DIKT,FW,Lot} and references therein.  Related notions of double boundary layers also appear in fluid dynamics and perturbation theory, see \cite{C,KK} for example.  As such, it seems natural to consider such a double boundary condition in the context of BV-energy minimizers.      

By setting $\psi_2\equiv\infty,$ and $f=g,$ we recover the classical obstacle problem for least gradients.  If in addition we set $\psi_1\equiv-\infty,$ then we recover the Dirichlet problem for least gradient functions (also referred to as the least gradient problem).  As mentioned above, the least gradient problem was first studied in the Euclidean setting by Sternberg, Williams, and Ziemer \cite{SWZ}, who showed that unique solutions exist for continuous boundary data, provided that the boundary of the domain has nonnegative mean curvature and is not locally area-minimizing.

Since its introduction in \cite{SWZ}, existence, uniqueness, and regularity of the least gradient problem above have been studied extensively in the Euclidean setting; for a sampling, see \cite{G2,G3,G1,GRS,JMN,K,MRL,Mor,MNT,NTT,RS,RS1,ST,Z} and the references therein.  In particular, weighted versions of this problem have applications to current density impedance imaging, see for example \cite{Mor,MNT,NTT}.

In recent decades, analysis in metric spaces has become a field of active study, in particular when the space is equipped with a doubling measure and supports a Poincar\'e inequality, see for example \cite{AGS,BB,H,HKST}.  A definition of BV functions and sets of finite perimeter was extended to this setting by Miranda Jr.\ in \cite{M}, and consequently, a theory of least gradient functions in metric spaces has been developed in recent years, see for example \cite{GM,GM1,HKLS,KKLS,K,KLLS,LMSS,LS,MSS}.  In \cite{LMSS}, Lahti, Mal\'y, Shanmugalingam, and Speight studied the Dirichlet problem for least gradients, originally introduced in \cite{SWZ}, in the metric setting.  To do so, they introduced a definition of positive mean curvature in the metric setting (Definition~\ref{defn:PMC} below) and showed that solutions exist for continuous boundary data if the boundary of the domain satisfies this condition.  It is this curvature condition that we assume in Theorem~\ref{thm:ExistStrongSoln}.  

In \cite{LMSS} it was also shown that in the weighted unit disc, uniqueness and continuity of solutions may fail even for Lipschitz boundary data.  As the problem we study in this paper is a generalization of this Dirichlet problem, we cannot guarantee uniqueness or continuity of strong solutions to the $\K_{\psi_1,\psi_2,f,g}$-obstacle problem.  For this reason, our comparison-type result, Theorem~\ref{cor:DObsCompThm}, is stated in terms of minimal and maximal strong solutions. For more on existence, uniqueness, and regularity of solutions to the weighted least gradient problem, see \cite{G1,JMN,Z}.

To construct our solution, we adapt the program first implemented by Sternberg, Williams, and Ziemer to obtain solutions to the Dirichlet problem for least gradients in the Euclidean setting in \cite{SWZ}.  There they constructed a least gradient solution by first solving the Dirichlet problem for boundary data consisting of superlevel sets of the original boundary data.  This method relies upon the fact, first discovered by Bombieri, De Giorgi, and Giusti \cite{BDG}, that characteristic functions of superlevel sets of least gradient functions are themselves of least gradient.  The construction of solutions to the obstacle problem for least gradients in \cite{ZZ} is an adaptation of the program from \cite{SWZ}.

In the metric setting, the construction of solutions to the Dirichlet problem for least gradients given in \cite{LMSS} is also inspired by the method established in \cite{SWZ}.  However both  \cite{SWZ} and \cite{ZZ} utilize smoothness properties and tangent cones for the boundaries of certain solution sets, tools which are not available in the metric setting.  Thus the construction in \cite{LMSS} is a further modification of the method from  \cite{SWZ}.  As we are studying the double obstacle problem with double boundary condition in the metric setting, our construction is inspired by that of \cite{LMSS}.

In \cite{LMSS}, the authors defined weak solutions to the Dirichlet problem for least gradients (Definition~\ref{defn:WeakSolnDP} below) which are easily obtained via the direct method of calculus of variations for a large class of boundary data.  They then used these weak solutions to construct their strong solution.  Since we have introduced a double boundary condition, and thus competitor functions are not fixed outside the domain, it is difficult to obtain weak solutions of this form by the direct method.  For this reason, we define a family of $\eps$-weak solutions (Definition~\ref{defn:WeakEpsSoln} below), which consider the BV-energy in slight enlargements of the domain $\Omega.$  By controlling these $\eps$-weak solutions as $\eps\to 0,$ we obtain the proper building blocks from which to construct our strong solution to the original problem in the manner of \cite{LMSS}.  Since our argument involves enlargements of $\Omega$, we assume that $\Omega$ is a uniform domain in order to use the strong BV extension results obtained in \cite{L1}. 

The structure of our paper is as follows: in Section 2, we introduce the basic definitions, notations, and assumptions used throughout the paper.  In Section 3, we prove preliminary results regarding solutions to the double obstacle, double boundary problem when the obstacle and boundary functions are characteristic functions of certain open sets.  In Section 4, we use these preliminary results to prove Theorem~\ref{thm:ExistStrongSoln} and Theorem~\ref{cor:DObsCompThm}.  In Section~5, we provide examples illustrating how, in the absence of obstacles, strong solutions to the double boundary problem may not be solutions to the Dirichlet problems for either boundary condition.  In this section, we also study how solutions of the double obstacle, double boundary problem converge to solutions of the double obstacle, single boundary problem as the  double boundary data converge to a single datum (see Theorem~\ref{thm:Stability} below).
\\
\\
\noindent{\bf Acknowledgments.}  The author was partially supported by the NSF Grant \#DMS-2054960 and the Taft Research Center Graduate Enrichment Award.  The author would like to thank Nageswari Shanmugalingam for her kind encouragement and many helpful discussions regarding this project.

\section{Background}
\subsection{General metric measure spaces}
Throughout this paper, we assume that $(X,d,\mu)$ is a complete metric measure space, with $\mu$ a doubling Borel regular measure supporting a $(1,1)$-Poincar\'e inequality (defined below).  By \emph{doubling}, we mean that there exists a constant $C\ge 1$ such that 
$$0<\mu(B(x,2r))\le C\mu(B(x,r))<\infty$$
for all $x\in X$ and $r>0.$  By iterating this condition, there exist constants $C_D\ge 1$ and $Q>1$ such that 
\begin{equation}\label{eq:LMB}
\frac{\mu(B(y,r))}{\mu(B(x,R))}\ge C_D^{-1}\left(\frac{r}{R}\right)^Q
\end{equation}
for every $0<r\le R$, $x\in X$, and $y\in B(x,R).$  In this paper we let $C$ denote constants, depending, unless otherwise stated, only on $\Omega$ and the doubling and Poincar\'e inequality constants (see below), whose precise value is not needed.  The value of $C$ may differ even within the same line.

Complete metric spaces equipped with doubling measures are necessarily proper, i.e.\ closed and bounded sets are compact.  For any open set $\Omega\subset X$, we define $L^1_\loc(\Omega)$ as the space of functions that are in $L^1(\Omega')$ for every $\Omega'\Subset\Omega,$ that is, for every open set $\Omega'$ such that $\overline{\Omega'}$ is a compact subset of $\Omega.$  Other local function spaces are defined analogously.  For a ball $B=B(x,r)$ and $\lambda>0$, we often denote $\lambda B:=B(x,\lambda r).$  If $A,B\subset X,$ then by $A\sqsubset B$, we mean that $\mu(A\setminus B)=0.$  By a \emph{domain}, we mean a nonempty connected open set in $X.$ 

We say that $X$ is $A$-\emph{quasiconvex} if $A\ge 1$ and for all points $x,y\in X,$ there exists a curve $\gamma$ connecting $x$ and $y$ such that 
$$\ell(\gamma)\le Ad(x,y).$$
We say that $\Omega$ is an $A$-\emph{uniform domain} if $A\ge 1$ and for all $x,y\in\Omega,$ there exists a curve $\gamma$ in $\Omega$ connecting $x$ and $y$ such that $\ell(\gamma)\le Ad(x,y)$ and 
$$\min\{\ell(\gamma_{x,z}),\ell(\gamma_{z,y})\}\le A\dist(z,X\setminus\Omega)$$
for all $z\in\gamma.$  By $\gamma_{x,z}$, we mean any subcurve of $\gamma$ joining $x$ to $z.$  

\subsection{Newtonian spaces and BV functions}

For a function $u:X\to\overline\R,$ we say that a Borel function $g:X\to[0,\infty]$ is an \emph{upper gradient} of $u$ if the following inequality holds for all non-constant compact rectifiable curves $\gamma:[a,b]\to X,$ 
$$|u(x)-u(y)|\le\int_\gamma gds,$$
whenever $u(x)$ and $u(y)$ are both finite, and $\int_\gamma gds=\infty$ otherwise.  Here $x$ and $y$ denote the endpoints of the curve $\gamma.$  Upper gradients were originally introduced in \cite{HK}.

Let $\widetilde N^{1,1}(X)$ be the class of all functions in $L^1(X)$ for which there exists an upper gradient in $L^1(X).$  For $u\in\widetilde N^{1,1}(X),$ we define 
$$\|u\|_{\widetilde N^{1,1}(X)}=\|u\|_{L^1(X)}+\inf_g\|g\|_{L^1(X)},$$
where the infimum is taken over all upper gradients $g$ of $u$.  Now, we define an equivalence relation in $\widetilde N^{1,1}(X)$ by $u\sim v$ if and only if $\|u-v\|_{\widetilde N^{1,1}(X)}=0.$  

The \emph{Newtonian space} $N^{1,1}(X)$ is defined as the quotient $\widetilde N^{1,1}(X)/\sim,$ and it is equipped with the norm $\|u\|_{N^{1,1}(X)}=\|u\|_{\widetilde N^{1,1}(X)}.$  One can analogously define $N^{1,1}(\Omega)$ for an open set $\Omega\subset X.$  When $X=\R^n$ or when $X$ is a Carnot group, we have that $N^{1,1}(X)=W^{1,1}(X).$  For more on Newtonian spaces, see \cite{S} or \cite{BB,Haj,HKST}.

To define functions of bounded variation on metric spaces, we follow the definition introduced by Miranda Jr.\ in \cite{M}.  For an open set $\Omega\subset X$ and a function $u\in L^1_\loc(\Omega),$ we define the \emph{total variation} of $u$ in $\Omega$ by 
$$\|Du\|(\Omega)=\inf\left\{\liminf_{i\to\infty}\int_\Omega g_{u_i}d\mu: N^{1,1}_\loc(\Omega)\ni u_i\to u\text{ in }L^1_\loc(\Omega)\right\},$$
where $g_{u_i}$ are upper gradients of $u_i.$  For an arbitrary set $A\subset X,$ we define 
$$\|Du\|(A)=\inf\left\{\|Du\|(\Omega):A\subset\Omega,\Omega\subset X\text{ open}\right\}.$$
We say that a function $u\in L^1(\Omega)$ is in $BV(\Omega)$ (i.e.\ of \emph{bounded variation} in $\Omega$) if $\|Du\|(\Omega)<\infty.$  We equip $BV(\Omega)$ with the norm 
$$\|u\|_{BV(\Omega)}=\|u\|_{L^1(\Omega)}+\|Du\|(\Omega).$$
This definition of of the BV class coincides with the standard definition in the Euclidean setting, see \cite{AFP,EG}.  For more on the BV class in metric spaces, see \cite{A,AMP}. 

It was shown in \cite[Theorem~3.4]{M} that if $u\in BV(X)$, then $\|Du\|(\cdot)$ is a finite Radon measure on $X.$  Furthermore, the BV energy is lower semi-continuous with respect to $L^1$ convergence.  That is, for an open set $\Omega\subset X$, if $u_k\to u$ in $L^1_\loc(\Omega),$ then 
$$\|Du\|(\Omega)\le\liminf_{k\to\infty}\|Du_k\|(\Omega).$$
This property is crucial to prove existence of the minimizers considered in this paper.  

A measurable set $E\subset X$ is of \emph{finite perimeter} if $\|D\chi_E\|(X)<\infty,$ and we denote the \emph{perimeter} of $E$ in $\Omega$ by 
$$P(E,\Omega)=\|D\chi_E\|(\Omega).$$
It was also shown in \cite[Proposition~4.2]{M} that the following coarea formula holds: if $\Omega\subset X$ is open and $u\in L^1_\loc(\Omega),$ then 
$$\|Du\|(\Omega)=\int_\R P(\{u>t\},\Omega)dt.$$
If $u\in BV(X),$ then this holds with $\Omega$ replaced by any Borel set $A\subset\Omega.$  

\subsection{Poincar\'e inequality and consequences}
Throughout this paper, we assume that $X$ supports a $(1,1)$-\emph{Poincar\'e inequality}.  That is, there exist constants $\lambda, C_I\ge 1$ such that for every ball $B$, every locally integrable function $u$, and every upper gradient $g$ of $u$, we have that 
$$\fint_B|u-u_B|d\mu\le C_I\rad(B)\fint_{\lambda B}g\,d\mu.$$
Here and throughout the paper, $$u_B=\fint_B u\,d\mu=\frac{1}{\mu(B)}\int_B u\,d\mu.$$

It was shown in \cite{KL} that the $(1,1)$-Poincar\'e inequality is equivalent to the \emph{relative isoperimetric inequality}: given $E\subset X,$
$$\min\{\mu(B\cap E),\mu(B\setminus E)\}\le C\rad(B)P(E,\lambda B)$$ 
for each ball $B\subset X.$  

Given $E\subset X$, we define its \emph{codimension $1$ Hausdorff measure} by 
\begin{equation}\label{eq:HaMeasure}
\Ha(E)=\lim_{\delta\to 0^+}\inf\left\{\sum_i\frac{\mu(B_i)}{\rad(B_i)}:B_i\text{ balls in }X,\, E\subset\bigcup_i B_i,\,\rad(B_i)<\delta\right\}.
\end{equation} 
We say that $\Ha|_{\partial\Omega}$ is \emph{lower codimension $1$ Ahlfors regular} if there exists $C>0$ such that 
\begin{equation}\label{eq:LowerAR}\Ha(B(x,r)\cap\partial\Omega)\ge C\frac{\mu(B(x,r))}{r}\end{equation}
for every $x\in\partial\Omega$ and $0<r<2\diam(\partial\Omega).$

In \cite{A,AMP}, it was shown that if $\mu$ is doubling and $X$ supports a $(1,1)$-Poincar\'e inequality, then there exists a constant $C\ge 1$ such that whenever $E\subset X$ is of finite perimeter and $A\subset X$ is a Borel set, we have 
$$C^{-1}\Ha(A\cap\partial_M E)\le P(E,A)\le C\Ha(A\cap\partial_M E).$$ 
Here $\partial_M E$ is the \emph{measure-theoretic boundary} of $E,$ which is the set of all points $x\in X$ such that 
$$\limsup_{r\to 0^+}\frac{\mu(B(x,r)\cap E)}{\mu(B(x,r))}>0\quad\text{and}\quad\limsup_{r\to 0^+}\frac{\mu(B(x,r)\setminus E)}{\mu(B(x,r))}>0.$$

For an extended real-valued function $u$ on $X$, we define the approximate upper and lower limits of $u$ by 
\begin{align*}
u^\vee(x)&=\inf\left\{t\in\R:\lim_{r\to 0^+}\frac{\mu(\{u>t\}\cap B(x,r))}{\mu(B(x,r))}=0\right\},\\
u^\wedge(x)&=\sup\left\{t\in\R:\lim_{r\to 0^+}\frac{\mu(\{u<t\}\cap B(x,r))}{\mu(B(x,r))}=0\right\}.
\end{align*}
From the Lebesgue differentiation theorem, $u^\vee=u^\wedge$ $\mu$-a.e.\ if $u\in L^1_\loc(X).$  The points for which $u^\wedge(x)=u^\vee(x)$ are said to be \emph{points of approximate continuity} of $u.$  We denote by
$$S_u=\{u^\wedge<u^\vee\},$$
the so called \emph{jump set} of $u$.  If $u\in BV(X)$, then by \cite[Proposition 5.2]{AMP}, $S_u$ is of $\sigma$-finite $\Ha$-measure.  Furthermore, by \cite[Theorem~5.3]{AMP}, the Radon measure $\|Du\|(\cdot)$ has the following decomposition:
\begin{equation}\label{eq:Decomp}
d\|Du\|=g\,d\mu+d\|D^ju\|+d\|D^cu\|.
\end{equation}  
Here, $g\in L^1(X)$ and $g\,d\mu$ is the part of $\|Du\|$ that is absolutely continuous with respect to $\mu$.  We call $\|D^ju\|$ the \emph{jump-part} of $u$, which is supported in $S_u$ and is absolutely continuous with respect to $\Ha|_{S_u}.$  We call the measure $\|D^cu\|$ the \emph{Cantor-part} of $\|Du\|$.  This measure does not detect sets of $\sigma$-finite $\Ha$-measure.

The following Leibniz rule for BV functions was shown in \cite[Theorem~4.7]{L3}. If $\Omega\subset X$ is open and $u,v\in L^1_\loc(\Omega)$, then
\begin{equation}\label{eq:Leibniz}
\|D(uv)\|(\Omega)\le\int_\Omega|u|^\vee d\|Dv\|+\int_{\Omega}|v|^\vee d\|Du\|.
\end{equation}

\subsection{Least gradient functions and obstacle problems}

\begin{defn}
Let $\Omega\subset X$ be open, and let $u\in BV_\loc(\Omega).$  We say that $u$ is of \emph{least gradient} in $\Omega$ if 
$$\|Du\|(V)\le\|Dv\|(V),$$
whenever $v\in BV(\Omega)$ with $\overline{\{x\in\Omega: u(x)\ne v(x)\}}\subset V\Subset\Omega.$
\end{defn}

\begin{defn}\label{defn:WeakSolnDP}
Let $\Omega$ be a bounded domain in $X$ with $\mu(X\setminus\Omega)>0,$ and let $f\in BV_\loc(X).$  We say that $u\in BV_\loc(X)$ is a \emph{weak solution to the Dirichlet problem for least gradients in $\Omega$ with boundary data $f$}, or simply, \emph{weak solution to the Dirichlet problem with boundary data $f$}, if $u=f$ on $X\setminus\Omega$ and 
$$\|Du\|(\overline\Omega)\le\|Dv\|(\overline\Omega)$$
whenever $v\in BV(X)$ with $v=f$ on $X\setminus\Omega.$  
\end{defn}

Given a bounded domain $\Omega\subset X$ and a function $u$ on $\Omega,$ we say that $u$ has a \emph{trace} at a point $x\in\partial\Omega$ if there is a number $Tu(x)\in\R$ such that 
$$\lim_{r\to 0^+}\fint_{B(x,r)\cap\Omega}|u-Tu(x)|d\mu=0.$$
We record the following lemma regarding traces of superlevel sets.  For proof, see \cite{K} for example.

\begin{lem}\label{lem:SupLevelTrace}\emph{(\cite[Lemma~4.6]{K})}
Let $f\in L^1(\partial\Omega),$ and suppose that $u\in L^1(\Omega)$ is such that $Tu=f$ $\Ha$-a.e.\ on $\partial\Omega.$  Then for $\Leb$-a.e.\ $t\in\R$, we have that $T\chi_{\{u>t\}}=\chi_{\{f>t\}}$ $\Ha$-a.e.\ on $\partial\Omega.$ 
\end{lem}

\begin{defn}
Let $\Omega$ be a domain in $X$ and $f:\partial\Omega\to\R.$  We say that a function $u\in BV(\Omega)$ is a \emph{strong solution to the Dirichlet problem for least gradients in $\Omega$ with boundary data $f$}, or simply, \emph{strong solution to the Dirichlet problem with boundary data $f$}, if $Tu=f$ $\Ha$-a.e.\ on $\partial\Omega$ and whenever $v\in BV(\Omega),$ with $Tv=f$ $\Ha$-a.e.\ on $\partial\Omega,$ we must have 
$$\|Du\|(\Omega)\le\|Dv\|(\Omega).$$ 
\end{defn}

Weak and strong solutions to Dirichlet problems on a domain $\Omega$ are necessarily of least gradient in $\Omega.$    

\begin{defn}\label{defn:MinMax}
A weak (strong) solution $u$ to the Dirichlet problem with boundary data $f$ is called the \emph{minimal} weak (strong) solution if every weak (strong) solution $u'$ to the Dirichlet problem with boundary data $f$ satisfies $u\le u'$ $\mu$-a.e.\ in $X$ (or in $\Omega$ in the case of strong solutions).  Likewise, a weak (strong) solution $u$ is called the \emph{maximal} weak (strong) solution if every weak (strong) solution $u'$ to the Dirichlet problem satisfies $u\ge u'$ $\mu$-a.e.\ in $X$ (or in $\Omega$ in the case of strong solutions).  We note that minimal and maximal weak and strong solutions are necessarily unique $\mu$-a.e.\
\end{defn}

In \cite{LMSS}, it is shown that for a set $F$ of finite perimeter, there always exists $E\subset X$ such that $\chi_E$ is a weak solution to the  Dirichlet problem with boundary data $\chi_F.$  We call $E$ a \emph{weak solution set}.  Moreover, for such $F$ there exists a unique minimal weak solution set, denoted $E_F.$  However without additional assumptions on the domain, it is not possible to guarantee that a weak solution set will be a \emph{strong solution set}; that is, $\chi_E$ may not be a strong solution to the Dirichlet problem with boundary data $\chi_F|_{\partial\Omega}$.  There exist boundary data for which weak solutions exist, but no strong solution exists.  For this reason, the following definition was introduced in \cite{LMSS}.  

\begin{defn}\label{defn:PMC}
Given a domain $\Omega\subset X$, we say that the boundary $\partial\Omega$ has \emph{positive mean curvature} if for each $x\in\partial\Omega,$ there exists a non-decreasing function $\phi_x:(0,\infty)\to(0,\infty)$ and a constant $r_x>0$ such that for all $0<r<r_x$ with $P(B(x,r),X)<\infty$, we have that $B(x,\phi_x(r))\sqsubset E_{B(x,r)}$, where $E_{B(x,r)}$ is the minimal weak solution set to the Dirichlet problem with boundary data $\chi_{B(x,r)}$.
\end{defn}  

In \cite{LMSS}, it was shown that if $\Omega$ has boundary of positive mean curvature and $\Ha(\partial\Omega)<\infty,$ then for every open set $F\subset X$ of finite perimeter such that $\Ha(\partial\Omega\cap\partial F)=0,$ all weak solutions to the Dirichlet problem with boundary data $\chi_F$ are strong solutions.  Moreover, if $v\in BV(\Omega)$ is a strong solution, then extending $v$ outside $\Omega$ by $\chi_F$ yields a weak solution.  This phenomenon also holds for some other classes of boundary data.  In \cite[Proposition~4.13]{LMSS}, it was shown that if $\Omega$ has boundary of positive mean curvature with $\Ha(\partial\Omega)<\infty$ and if $f\in BV_\loc(X)\cap C(X),$ then weak solutions to the Dirichlet problem with boundary data $f$ are strong solutions, and strong solutions extended outside $\Omega$ by $f$ are weak solutions.

In this paper, we study a generalization of the Dirichlet problem for least gradients, namely by adding obstacle functions in the domain and by loosening the boundary condition to allow a range of admissible boundary values.  We define weak and strong solutions for this modified problem as follows.  

\begin{defn}\label{defn:WeakEpsSoln}
Let $\Omega\subset X$ be a domain with $\mu(X\setminus\Omega)>0,$ and let $f,g\in L^1_\loc(X).$ Additionally let $\psi_1,\psi_2:\overline\Omega\to\overline\R,$ and let 
$$\wtil\K_{\psi_1,\psi_2,f,g}(\Omega):=\{u\in L^1_\loc(X):\psi_1\le u\le\psi_2\text{ in }\Omega,\;f\le u\le g\text{ in }X\setminus\Omega\},$$
where the inequalities are in the pointwise {\it everywhere} sense.
For $\eps>0,$ let 
$$\Omega_\eps:=\{x\in X:\dist(x,\Omega)<\eps\}.$$
We say that $u\in\wtil\K_{\psi_1,\psi_2,f,g}(\Omega)$ is an {\it $\eps$-weak solution to the $\K_{\psi_1,\psi_2,f,g}$-obstacle problem} if, for all $v\in\wtil\K_{\psi_1,\psi_2,f,g}(\Omega),$ we have that 
$$\|Du\|(\Omega_\eps)\le\|Dv\|(\Omega_\eps).$$ 
\end{defn}

\begin{remark}
From the definitions, we see that if $0<\eps_1\le\eps_2$ and $u_1$ and $u_2$ are $\eps_1$ and $\eps_2$-weak solutions respectively, then it follows that 
$$\|Du_1\|(\Omega_{\eps_1})\le\|Du_2\|(\Omega_{\eps_1}).$$
Furthermore if $f=g$, then we have that $\|Dv_1\|(X\setminus\overline\Omega)=\|Dv_2\|(X\setminus\overline\Omega)$ for all $v_1,v_2\in\wtil\K_{\psi_1,\psi_2,f,g}(\Omega)$.  As such, if $u$ is an $\eps_0$-weak solution for some $\eps_0>0,$ then $u$ is an an $\eps$-weak solution for all $\eps>0.$  Hence in the case where $f=g$, we say that $u$ is a \emph{weak solution} to the $\K_{\psi_1,\psi_2,f,g}$-obstacle problem if 
$$\|Du\|(\overline\Omega)\le\|Dv\|(\overline\Omega)$$ for all $v\in\wtil\K_{\psi_1,\psi_2,f,g}(\Omega).$

\end{remark}

\begin{defn}\label{defn:StrongSolnDB}
Let $\Omega$ be a bounded domain, and let $f,g\in L^1(\partial\Omega).$ Additionally, let $\psi_1,\psi_2:\overline\Omega\to\overline\R$, and let 
$$\K_{\psi_1,\psi_2,f,g}(\Omega):=\{u\in BV(\Omega):\psi_1\le u\le\psi_2\text{ in }\Omega,\;f\le Tu\le g\text{ on }\partial\Omega\}$$
where the inequalites are in the $\mu$ and $\Ha$-a.e.\ sense respectively.  
We say that $u\in\K_{\psi_1,\psi_2,f,g}(\Omega)$ is a {\it strong solution to the $\K_{\psi_1,\psi_2,f,g}$-obstacle problem} if 
$$\|Du\|(\Omega)\le\|Dv\|(\Omega)$$
 for all $v\in\K_{\psi_1,\psi_2,f,g}(\Omega).$ 
\end{defn}
In the case where $\psi_1=\chi_A,$ $\psi_2=\chi_B,$ $f=\chi_F$ and $g=\chi_G,$ we set 
$$\K_{A,B,F,G}:=\K_{\chi_A,\chi_B,\chi_F,\chi_G}$$ in either of the above definitions for ease of notation.  

While the inequalities in the definition of $\K_{\psi_1,\psi_2,f,g}(\Omega)$ above are in the almost everywhere sense, we note that the inequalities in the definition of $\wtil\K_{\psi_1,\psi_2,f,g}(\Omega)$ are in the \emph{everywhere} sense.  Since these functions pertain to  $\eps$-weak solutions, this will not be too restrictive for our purposes.

Note that by taking $\psi_1\equiv-\infty$ and $\psi_2\equiv\infty$ and by letting $f=g$, we recover the standard Dirichlet problem for least gradients as a special case of this formulation.  Likewise, we obtain any combination of obstacle problem with upper or lower obstacles or boundary data.  Moreover, we can define minimal and maximal strong solutions in the same manner as Definition~\ref{defn:MinMax}:

\begin{defn}
A strong solution $u$ to the $\K_{\psi_1,\psi_2,f,g}$-obstacle problem is called the \emph{minimal} (\emph{maximal}) strong solution if any strong solution $u'$ satisfies $u\le u'$ ($u\ge u'$) $\mu$-a.e.\ in $\Omega.$  We note that minimal and maximal strong solutions are necessarily unique $\mu$-a.e.
\end{defn}

If $E\subset X$ is a weak solution set to the Dirichlet problem with boundary data $\chi_F$ for some $F\subset X,$ then $E$ is a $1$-quasiminimal set as in \cite[Definition~3.1]{KKLS}.  Thus by \cite[Theorem~4.2]{KKLS}, $E$ satisfies a uniform measure density condition inside $\Omega$.  That is, by modifying $E$ on a set of measure zero if necessary, there exists $\gamma_0>0$ depending only on the doubling and Poincar\'e inequality constants such that for all $x\in\Omega\cap\partial E,$ 
$$\frac{\mu(B(x,r)\cap E)}{\mu(B(x,r))}\ge\gamma_0\quad\text{and}\quad\frac{\mu(B(x,r)\setminus E)}{\mu(B(x,r))}\ge\gamma_0$$
whenever $0<r<\diam(X)/3$ such that $B(x,2r)\subset\Omega.$  We refer to this modified $E$ as a \emph{regularized} weak solution set.  Furthermore, it was shown in \cite[Theorem~5.2]{KKLS} that such regularized weak solution sets are uniformly locally porous.

\begin{thm}\label{thm:Porosity}\emph{(\cite[Theorem~5.2]{KKLS})}
There exists a constant $C_P\ge 1$ such that for every set $E$ which is $1$-quasiminimal set in $\Omega$ and for every $x\in\Omega\cap\partial E$ and $r>0$ such that $B(x,10r)\subset\Omega$, there exist points $y,z\in B(x,r)$ such that 
$$B(y,r/C_P)\subset E\cap\Omega\quad\text{and}\quad B(z,r/C_P)\subset \Omega\setminus E.$$
The constant $C_P$ depends only on the doubling and Poincar\'e constants.
\end{thm} 

We conclude this section by using the above porosity property to prove the following useful lemma, which gives us a bound on the size of certain regularized minimal weak solutions to Dirichlet problems.

\begin{lem}\label{lem:SolnSetBound}
Let $\Omega$ be a bounded domain with $\mu(X\setminus\Omega)>0$, and suppose that the following hold:
\begin{enumerate}
\item[(i)] $\Ha|_{\partial\Omega}$ is lower codimension 1 Ahlfors regular \eqref{eq:LowerAR}, and $\Ha(\partial\Omega)<\infty$.
\item[(ii)] $\Ha(\{z\})=0$ for all $z\in\partial\Omega,$
\item[(iii)] $\partial\Omega$ satisfies the positive mean curvature condition (Definition~\ref{defn:PMC}).
\end{enumerate}
Then for all $z\in\partial\Omega$ and $R>0$ such that $X\setminus B(z,4R)\ne\varnothing$, there exists $0<r_0<R$ such that $E_{B(z,r_0)}\subset B(z,R),$ where $E_{B(z,r_0)}$ is the regularized minimal weak solution to the Dirichlet problem with boundary data $\chi_{B(z,r_0)}.$    
\end{lem}

\begin{proof}
Let $z\in\partial\Omega$ and $R>0$ such that $X\setminus B(z,4R)$ is nonempty.  Towards a contradiction, suppose that $E_{B(z,r)}\setminus B(z,R)\ne\varnothing$ for all $0<r<R$, where $E_{B(z,r)}$ is the regularized minimal weak solution to the Dirichlet problem with boundary data $\chi_{B(z,r)}.$  By \cite[Lemma~6.2]{KKST}, for each $0<r<R,$ there exists $\rho\in[r/2,r]$ such that $B(z,\rho)$ is of finite perimeter in $X$ and $$P(B(z,\rho),X)\simeq\frac{\mu(B(z,\rho)}{\rho}.$$  Therefore, we can choose a sequence $0<r_k<R/4$ such that $r_k\to 0$ as $k\to\infty,$ where $P(B(z,r_k),X)<\infty$, $\Ha(\partial\Omega\cap\partial B(z,r_k))=0,$ and 
$$P(B(z,r_k),X)\simeq\frac{\mu(B(z,r_k))}{r_k}.$$ 
By lower codimension 1 Ahlfors regularity, we have that 
\[
P(B(z,r_k),X)\le C\Ha(B(z,r_k)\cap\partial\Omega),
\]
and since $\Ha(\{z\})=0,$ it follows that 
\begin{equation*}
\lim_{k\to\infty}P(B(z,r_k),X)=0.
\end{equation*}

Letting $f_k:=\chi_{B(z,r_k)},$ we have that $f_k\to 0$ in $BV(X).$  Letting $E_k:=E_{B(z,r_k)}$ be the regularized minimal weak solution set to the Dirichlet problem with boundary data $f_k$, it follows from \cite[Proposition~3.3]{HKLS} that there exists a set $E\subset X$ of finite perimeter such that $\chi_{E_k}\to \chi_E$ in $L^1(\Omega)$ by passing to a subsequence if necessary.  In fact, the sequence converges pointwise monotonically a.e.\ as well (see \cite[Lemma~3.8]{LMSS}).  Moreover, $\chi_E$ is a weak solution to the Dirichlet problem with zero boundary values.  Hence we have that $P(E,X)=0,$ and from the relative isoperimetric inequality it follows that $\mu(E)=0.$  Thus, we have that 
\begin{equation}\label{eq:MeasToZero}
\lim_{k\to\infty}\mu(E_k)=0.
\end{equation}  

  For each $y\in\partial\Omega\setminus B(z,R/2),$ there exists $r>0$ such that $P(B(y,r),X)<\infty$ and $B(y,r)\cap B(z,r_k)=\varnothing$ for all $k\in\N.$  This is possible since $r_k<R/4$ for all $k\in\N.$  Noting that $1-\chi_{E_k}$ is a weak solution to the Dirichlet problem with boundary data $\chi_{X\setminus B(z,r_k)},$ it follows from \cite[Lemma~3.3]{LMSS} that $\mu(E_{B(y,r)}\cap E_k)=0,$ where $E_{B(y,r)}$ is the minimal weak solution to the Dirichlet problem with boundary data $\chi_{B(y,r)}.$  By positive mean curvature, there exists $r_y>0$ such that $B(y,r_y)\sqsubset E_{B(y,r)},$ hence $\mu(B(y,r_y)\cap E_k)=0.$  Thus, $B(y,r_y)$ contains no interior points of $E_k.$  If there exists $\zeta\in\partial E_k\cap B(y,r_y),$ then since $E_k$ is a regularized weak solution and thus satisfies the uniform measure density condition, there exists $\tau>0$ such that $B(\zeta,2\tau)\subset\Omega$ and  
$$0<\mu(B(\zeta,\tau)\cap E_k)\le \mu(B(y,r_y)\cap E_k),$$ a contradiction.  Therefore, we have that $E_k\cap B(y,r_y)=\varnothing.$

 By compactness, we can cover $\partial\Omega\setminus B(z,R/2)$ with a finite collection of balls $\{B(y_i,r_{y_i})\}_i$ such that $B(y_i,r_{y_i})\cap E_k=\varnothing$ for all $k\in\N.$  Since this covering is independent of $k$, it follows from a Lebesgue number argument that
 \[
\inf_k\dist(E_k,\partial\Omega\setminus B(z,R/2))>0.
\]
  
By our original supposition, there exists $x_k\in E_k\setminus B(z,R)$.  For each $k\in\N$, we have that $\dist(x_k,\partial\Omega\cap B(z,R/2))\ge R/2,$ and so it follows that 
\[
\inf_k\dist(x_k,\partial\Omega)>0.
\]
Thus the sequence $\{x_k\}_k$ is compactly contained in $\Omega,$ and so there exists $x_0\in\Omega$ and a subsequence, also denoted $\{x_k\}_k$, such that $x_k\to x_0$ as $k\to\infty.$ 

Since $X$ is complete and $\mu$ is doubling and supports a $(1,1)$-Poincar\'e inequality, it follows that $X$ is $A$-quasiconvex with $A\ge 1$ depending only on the doubling and Poincar\'e constants, see \cite{HKST} for example.  Choose $r_0>0$ sufficiently small so that $B(x_0,40Ar_0)\subset\Omega.$  Choose $k\in\N$ sufficiently large such that $x_k\in B(x_0,Ar_0)$ and 
\begin{equation}\label{eq:E_kBound}
\mu(E_k)<\frac{\mu(B(x_0,r_0))}{C_D(3C_PA)^Q},
\end{equation}
where $C_D$ and $Q$ are from \eqref{eq:LMB}, and $C_P$ is the constant from Theorem~\ref{thm:Porosity}.  This is possible by \eqref{eq:MeasToZero}.  By the doubling property, we have that for sufficiently large $k,$
$$\mu(B(x_k,r_0))\ge\frac{\mu(B(x_0,Ar_0))}{C_DA^Q}\ge\frac{\mu(B(x_0,r_0))}{C_D(3C_PA)^Q},$$
and so by \eqref{eq:E_kBound}, there exists $y\in B(x_k,r_0)\setminus E_k.$  Recall that $C_P\ge 1.$  

By quasiconvexity of $X,$ there exists a curve $\gamma$ joining $x_k$ to $y$ such that $$\ell(\gamma)\le Ad(x_k,y)\le Ar_0,$$ and so there exists $y_k\in\partial E_k\cap B(x_k,Ar_0).$  By our choices we have that $B(y_k,20Ar_0)\subset\Omega,$ and so by \cite[Theorem~5.2]{KKLS}, there exist points $x,x'\in B(y_k,Ar_0)$ such that 
$$B(x,Ar_0/C_P)\subset E_k\cap\Omega\quad\text{and}\quad B(x',Ar_0/C_P)\subset \Omega\setminus E_k.$$
Note that $x\in B(x_0,3Ar_0),$ and so from the doubling property and the fact that $A\ge 1$, we have that 
\begin{align*}
\mu(E_k)\ge\mu(B(x,Ar_0/C_P))\ge\frac{\mu(B(x_0,r_0))}{C_D(3C_P)^Q}\ge\frac{\mu(B(x_0,r_0))}{C_D(3C_PA)^Q}.
\end{align*}
This contradicts \eqref{eq:E_kBound}, completing the proof.   
\end{proof}

We note that the conclusion of Lemma~\ref{lem:SolnSetBound} may fail if one removes the assumption that all singletons on the boundary are $\Ha$-negligible, as shown by the following example:
\begin{example}\label{ex:Singletons}
Equip $\R$ with the measure $\mu(x)=\omega(x)d\Leb(x)$ where 
\[
\omega(x)=
\begin{cases}
\frac{1}{2}|x|+\frac{1}{2},&-1\le x\le 1\\
1,&1<|x|,
\end{cases} 
\]
and consider $\Omega=(-1,1)\subset\R.$  Then $\Ha(\{-1\})=\Ha(\{1\})=1,$ and $\Ha$ is lower codimension 1 Ahlfors regular on $\partial\Omega.$  Moreover, if we consider the ball centered at one of the boundary points, say $B(1,r)=(1-r,1+r),$ for any $0<r\le 2,$ we see that the minimal weak solution to the Dirichlet problem with boundary data $\chi_{B(1,r)}$ is the interval $(0,1+r).$  This is due to the fact that the weight $\omega$ attains a minimum at the origin, and so sets agreeing with $B(1,r)$ outside of $\Omega$ have smallest perimeter when they are intervals with a left endpoint at the origin.  Likewise, for $0<r\le 2,$ the minimal weak solution to the Dirichlet problem with boundary data $\chi_{B(-1,r)}$ is the interval $(-1-r,0).$  Hence, $\partial\Omega$ is of positive mean curvature in the sense of Definition~\ref{defn:PMC}, but the conclusion of Lemma~\ref{lem:SolnSetBound} fails.   
\end{example}

\section{Preliminary results for open set obstacles and data}

In this section, we assume that $\Omega\subset X$ is a bounded domain such that $\mu(X\setminus\Omega)>0$, $\Ha(\partial\Omega)<\infty$, $\Ha|_{\partial\Omega}$ is doubling, and that $\Ha|_{\partial\Omega}$ is lower codimension 1 Ahlfors regular as in \eqref{eq:LowerAR}. Furthermore, we assume that $\Ha(\{z\})=0$ for all $z\in\partial\Omega.$

In order to construct strong solutions to the $\K_{\psi_1,\psi_2,f,g}$-obstacle problem for continuous $\psi_1,\psi_2,f$ and $g$, we first construct strong solutions for the problem when obstacles and data consist of superlevel sets of the original obstacle and boundary functions.  To do this, we first show existence of certain weak solutions. However, since we consider a problem with a double boundary condition as opposed to a fixed single boundary condition, it becomes difficult to prove existence using a definition of weak solutions similar to Definition~\ref{defn:WeakSolnDP}.  For this reason, we begin by proving existence of weak $\eps$-solutions, as in Definition~\ref{defn:WeakEpsSoln}.    

In this section we prove results for obstacle sets $A,B$ and boundary sets $F,G$ which have very specific properties (see the hypotheses of Proposition~\ref{prop:StrongSolnSet}, for example).  As we shall see in Section 4, by assuming the hypothesis that $\K_{\psi_1,\psi,f,g}(\Omega)\ne\varnothing$ for $\psi_1,\psi_2\in C(\overline\Omega)$ and $f,g\in C(\partial\Omega)$ in Theorem~\ref{thm:ExistStrongSoln}, it follows that almost all of the superlevel sets of the obstacle and boundary functions will have the  properties used in this section.

\begin{lem}\label{lem:EpsWeakSoln}
Let $A,B\subset\overline\Omega$ and $F,G\subset X$ be Borel measurable, and suppose that $\wtil\K_{A,B,F,G}(\Omega)\cap BV_\loc(X)\ne\varnothing$.  Then for every $\eps>0$, there exists an $\eps$-weak solution to the $\K_{A,B,F,G}$-obstacle problem.  Furthermore, an $\eps$-weak solution exists of the form $u=\chi_{E_\eps}$ for some measurable set $E_\eps\subset X.$    
\end{lem}

We note that $\wtil\K_{A,B,F,G}(\Omega)\ne\varnothing$ implies that $A\subset B$ and $F\setminus\Omega\subset G\setminus\Omega.$

\begin{proof}
For each $\eps>0,$ let 
$$\lambda_\eps:=\inf\{\|Du\|(\Omega_\eps):u\in\wtil\K_{A,B,F,G}(\Omega)\}<\infty.$$  Since $\wtil\K_{A,B,F,G}(\Omega)\cap BV_\loc(X)\ne\varnothing$, we have that $\lambda_\eps<\infty.$  For each $k\in\N,$ let $u_k^\eps\in\wtil\K_{A,B,F,G}(\Omega)$ be such that $$\|Du_k^\eps\|(\Omega_\eps)\to\lambda_\eps$$ as $k\to\infty.$  Since truncation will not increase the BV energy, and since the obstacles and boundary conditions are characteristic functions of sets, without loss of generality, we may assume that $0\le u_k^\eps\le 1.$  Thus, we have that 
$$\sup_k\|u_k^\eps\|_{L^1(\Omega_\eps)}\le\mu(\Omega_\eps)<\infty,$$
and so $\{u_k^\eps\}_{k\in\N}$ is bounded in $BV(\Omega_\eps).$  Therefore by \cite[Theorem~3.7]{M}, there exists a subsequence (not relabeled) and a function $u_\eps\in BV(\Omega_\eps)$ such that $u_k^\eps\to u_\eps$ in $L^1(\Omega_\eps)$ and pointwise $\mu$-a.e.\ in $\Omega_\eps.$  By lower semicontinuity of the BV-energy, we have that 
$$\|Du_\eps\|(\Omega_\eps)\le\liminf_{k\to\infty}\|Du_k^\eps\|(\Omega_\eps)=\lambda_\eps.$$   
By modifying $u_\eps$ on a set of $\mu$-measure zero if necessary, hence not changing the BV-energy, we may assume that $\chi_A\le u_\eps\le\chi_B$ everywhere in $\Omega$ and $\chi_F\le u_\eps\le\chi_G$ everywhere in $\Omega_\eps\setminus\Omega.$  Extending $u_\eps$ outside $\Omega_\eps$ by $\chi_F,$ we have that $u_\eps\in\wtil\K_{A,B,F,G}(\Omega)$, and so it follows that $\|Du_\eps\|(\Omega_\eps)=\lambda_\eps.$  Thus, $u_\eps$ is an $\eps$-weak solution to the $\K_{A,B,F,G}$-obstacle problem.  Moreover, by the coarea formula we have that 
$$\|Du_\eps\|(\Omega_\eps)=\int_0^1 P(\{u_\eps>t\},\Omega_\eps)dt,$$ and so there exists $0<t<1$ such that $P(\{u_\eps>t\},\Omega_\eps)\le\|Du_\eps\|(\Omega_\eps).$  Therefore letting $E_\eps:=\{u_\eps>t\},$ we have that $\chi_A\le\chi_{E_\eps}\le\chi_B$ in $\Omega$ and $\chi_F\le\chi_{E_\eps}\le\chi_G$ in $X\setminus\Omega$ since $u_\eps\in\wtil\K_{A,B,F,G}(\Omega).$  Thus $\chi_{E_\eps}\in\wtil\K_{A,B,F,G}(\Omega),$ and so it follows that $u_\eps=\chi_{E_\eps}$ is an $\eps$-weak solution as well.
\end{proof}

Since we wish to prove that strong solutions exist when obstacles and boundary data consist of open sets, we need to ensure that the traces behave as we wish.  We will use the following lemma, which requires positive mean curvature of the boundary of the domain as in Definition~\ref{defn:PMC}.

\begin{lem}\label{lem:WeakSolnTrace}
Let $\Omega$ be a bounded domain with boundary of positive mean curvature.  Let $A,B\subset\overline\Omega$ be relatively open, let $F\subset G\subset X$ be open, and suppose that $\wtil\K_{A,B,F,G}(\Omega)\cap BV_\loc(X)\ne\varnothing$.  If $z\in\partial\Omega\cap F\cap B,$ then there exists $\rho_z>0$ such that for all $\eps>0$ and for all $\eps$-weak solution sets $E_\eps$ to the $\K_{A,B,F,G}$-obstacle problem, we have that $$\mu(B(z,\rho_z)\setminus E_\eps)=0.$$ 

Suppose in addition that $\chi_A\le\chi_G$ $\Ha$-a.e.\ on $\partial\Omega$ and that $\Ha(\partial\Omega\cap\partial_{\overline\Omega} A)=0,$ where $\partial_{\overline\Omega}A$ is the boundary of $A$ relative to $\overline\Omega.$  Then for $\Ha$-a.e.\ $z\in\partial\Omega\setminus\overline G,$ there exists $\rho_z>0$ such that for all $\eps>0$ and for all $\eps$-weak solution sets $E_\eps$ to the $\K_{A,B,F,G}$-obstacle problem, we have that 
$$\mu(B(z,\rho_z)\cap E_\eps)=0.$$
\end{lem}

\begin{proof}
We begin by proving the first assertion.  Let $z\in\partial\Omega\cap F\cap B.$  Then there exists $r>0$ such that $B(z,r)\subset F\subset G$ and $B(z,r)\cap\overline\Omega\subset B.$  By Lemma~\ref{lem:SolnSetBound}, there exists $0<r_0<r$ such that $E_0=E_{B(z,r_0)}\subset B(z,r)\subset F,$ where $E_0$ is the regularized minimal weak solution to the Dirichlet problem with boundary data $\chi_{B(z,r_0)}.$  By the positive mean curvature condition, there exists $\rho_z>0$ such that $B(z,\rho_z)\sqsubset E_{0}.$

Now let $\eps>0$ and let $E_\eps$ be an $\eps$-weak solution set to the $\K_{A,B,F,G}$-obstacle problem.  By \cite[Proposition~4.7]{M}, we have that 
$$P(E_0\cap E_\eps,\overline\Omega)+P(E_0\cup E_\eps,\overline\Omega)\le P(E_0,\overline\Omega)+P(E_\eps,\overline\Omega).$$
If $P(E_0\cap E_\eps,\overline\Omega)>P(E_0,\overline\Omega),$ then we have that $P(E_0\cup E_\eps,\overline\Omega)<P(E_\eps,\overline\Omega).$  In this case, 
\begin{align*}
P(E_0\cup E_\eps,\Omega_\eps)&=P(E_0\cup E_\eps,\overline\Omega)+P(E_0\cup E_\eps,\Omega_\eps\setminus\overline\Omega)\\
	&<P(E_\eps,\overline\Omega)+P(E_0\cup E_\eps,\Omega_\eps\setminus\overline\Omega).
\end{align*} 
Since $E_0\subset F$ and $\chi_{E_\eps}\in\wtil K_{A,B,F,G}(\Omega),$ we have that $E_0\setminus\overline\Omega\subset F\setminus\overline\Omega\subset E_\eps\setminus\overline\Omega.$  Therefore $(E_0\cup E_\eps)\setminus\overline\Omega=E_\eps\setminus\overline\Omega,$ and so it follows that 
\begin{align*}
P(E_0\cup E_\eps,\Omega_\eps)&<P(E_\eps,\overline\Omega)+P(E_\eps,\Omega_\eps\setminus\overline\Omega)=P(E_\eps,\Omega_\eps).
\end{align*}
This is a contradiction, since $E_\eps$ is an $\eps$-weak solution to the $\K_{A,B,F,G}$-obstacle problem and $\chi_{E_0\cup E_\eps}\in\wtil K_{A,B,F,G}(\Omega)$ by the choice of $r_0.$  Thus, we have that $P(E_0\cap E_\eps,\overline\Omega)\le P(E_0,\overline\Omega),$ and so $E_0\cap E_\eps$ is a weak solution to the Dirichlet problem with boundary data $B(z,r_0).$  Since $E_0$ is the minimal weak solution to the said Dirichlet problem, we have that $E_0\sqsubset (E_0\cap E_\eps),$ and so it follows that $E_0\sqsubset E_\eps.$  Since $B(z,\rho_z)\sqsubset E_0,$ we have that $$\mu(B(z,\rho_z)\setminus E_\eps)=0.$$ 

To prove the second assertion, let $z\in\partial\Omega\setminus\overline G$ such that $\dist(z,A)>0.$  Since $\Ha(\partial\Omega\cap\partial_{\overline\Omega} A)=0$ and since $\chi_A\le\chi_G$ $\Ha$-a.e.\ on $\partial\Omega,$ this holds for $\Ha$-a.e.\ $z\in\partial\Omega\setminus\overline G.$ Let $r>0$ be such that $B(z,r)\cap(\overline G\cup A)=\varnothing.$  As above, via Lemma~\ref{lem:SolnSetBound}, choose $r_0>0$ sufficiently small so that $E_0\subset B(z,r),$ where $E_0$ is the regularized minimal weak solution set to the Dirichlet problem with boundary data $\chi_{B(z,r_0)}.$  By positive mean curvature, there exists $\rho_z>0$ such that $B(z,\rho_z)\sqsubset E_0.$  

Now let $\eps>0$ and let $E_\eps$ be an $\eps$-weak solution to the $\K_{A,B,F,G}$-obstacle problem.  Again using \cite[Proposition 4.7]{M}, we have that 
$$P(E_0\setminus E_\eps,\overline\Omega)+P(E_\eps\setminus E_0,\overline\Omega)\le P(E_0,\overline\Omega)+P(E_\eps,\overline\Omega).$$  
We use the same argument as above.  If $P(E_0\setminus E_\eps,\overline\Omega)>P(E_0,\overline\Omega),$ then we would have that $P(E_\eps\setminus E_0,\Omega_\eps)<P(E_\eps,\Omega_\eps),$ a contradiction.  Therefore, it follows that $P(E_0\setminus E_\eps,\overline\Omega)\le P(E_0,\overline\Omega),$ and minimality of the weak solution $E_0$ implies that $\mu(E_0\cap E_\eps)=0.$  Therefore, since $B(z,\rho_z)\sqsubset E_0,$ it follows that \[
\mu(B(z,\rho_z)\cap E_\eps)=0.\qedhere
\]
\end{proof}

We are now ready to prove the main result of this section.  Note that we now assume that $\Omega$ is a uniform domain in order to accommodate extensions of functions in $BV(\Omega)$ to $BV(X)$ in the manner of \cite{L1}. 

\begin{prop}\label{prop:StrongSolnSet}
Let $\Omega$ be a bounded uniform domain with boundary of positive mean curvature.  Let $A,B\subset\overline\Omega$ be relatively open, let $F\subset G\subset X$ be open, and suppose that $\wtil\K_{A,B,F,G}(\Omega)\cap BV_\loc(X)\ne\varnothing$.  In addition, suppose that 
\begin{enumerate}
\item[(i)] $P(F,X),P(G,X)<\infty,$
\item[(ii)] $\Ha(\partial\Omega\cap(\partial_{\overline\Omega} A\cup\partial F\cup\partial G))=0,$
\item[(iii)] $\chi_F\le\chi_B$ and $\chi_A\le\chi_G$ $\Ha$-a.e.\ on $\partial\Omega.$
\end{enumerate}  
Then there exists a strong solution to the $\K_{A,B,F,G}$-obstacle problem. 
\end{prop}

\begin{proof}
For each $k\in\N$, let $\chi_{E_k}$ be a $1/k$-weak solution to the $\K_{A,B,F,G}$-obstacle problem.  The existence of such $1/k$-weak solutions is guaranteed by Lemma~\ref{lem:EpsWeakSoln}.  Then we have that 
$$P(E_k,\Omega_{1/k})=\lambda_{1/k}<\infty,$$ 
and by the definition of $1/k$-weak solutions, it follows that $\lambda_{1/(k+1)}\le\lambda_{1/k}$ for each $k\in\N.$  Therefore, it follows that
$$\sup_{k}P(E_k,\Omega)\le\sup_k P(E_k,\Omega_{1/k})<\infty,$$
and since $$\sup_k\|\chi_{E_k}\|_{L^1(\Omega)}\le\mu(\Omega)<\infty,$$
we have that $\{\chi_{E_k}\}_{k\in\N}$ is bounded in $BV(\Omega).$  Hence, by \cite[Theorem~3.7]{M}, there exists a subsequence $\{\chi_{E_{k_j}}\}_{j\in\N}$ and a measurable set $E\subset\Omega$ such that $\chi_{E_{k_j}}\to\chi_E$ in $L^1(\Omega).$  Passing to a further subsequence if necessary (keeping the same notation), we also have pointwise convergence $\mu$-a.e.\ in $\Omega.$  By lower semicontinuity of the BV energy, it follows that 
\begin{equation}\label{*}
P(E,\Omega)\le\liminf_{j\to\infty}P(E_{k_j},\Omega)\le\liminf_{j\to\infty}\lambda_{1/k_j}<\infty,
\end{equation}
 and so $E$ is a set of finite perimeter in $\Omega.$  

Since each $\chi_{E_{k}}\in\wtil\K_{A,B,F,G}(\Omega),$ it follows from the pointwise $\mu$-a.e.\ convergence that $\chi_A\le\chi_E\le\chi_B$ $\mu$-a.e.\ in $\Omega.$  For $\Ha$-a.e.\ $z\in\partial\Omega\cap F,$ there exists $\rho_z>0$ such that $\mu(B(z,\rho_z)\setminus E_{k})=0$ for all $k\in\N$ by Lemma~\ref{lem:WeakSolnTrace}.  Therefore, again by the pointwise $\mu$-a.e.\ convergence, we have that $\mu(B(x,\rho_z)\cap\Omega\setminus E)=0.$  Hence $T\chi_E(z)=1=\chi_F(z).$  Likewise, by Lemma~\ref{lem:WeakSolnTrace}, we also have that $T\chi_E(z)=0=\chi_G(z)$ for $\Ha$-a.e.\ $z\in\partial\Omega\setminus\overline G.$  For $z\in\partial\Omega\cap G\setminus F,$ we necessarily have that $\chi_F(z)=0\le T\chi_E(z)\le 1=\chi_G(z).$  Since $\Ha(\partial\Omega\cap\partial G)=0$ by hypothesis, we therefore have that $\chi_F\le T\chi_E\le\chi_G$ $\Ha$-a.e.\ on $\partial\Omega,$ and so $\chi_E\in\K_{A,B,F,G}(\Omega).$

It remains to show that $\chi_E$ is a minimizer in the sense of Definition~\ref{defn:StrongSolnDB}.  To do this, let $v\in\K_{A,B,F,G}(\Omega)$ and suppose that $P(E,\Omega)>\|Dv\|(\Omega).$  Since truncation does not increase BV-energy, we may assume that $0\le v\le 1.$   By modifying $v$ on a set of $\mu$-measure zero if necessary, hence not changing the BV-energy, we may also assume that $\chi_A\le v\le\chi_B$ everywhere in $\Omega.$  Since $\Omega$ is a uniform domain, it follows from \cite[Theorem~3.1]{L1} that there exists an extension $Ev\in BV(X)$ such that $Ev|_\Omega=v$ and $\|DEv\|(\partial\Omega)=0,$ and again by truncation we may assume that $0\le Ev\le 1.$  Now, let $\til v:X\to\R$ be given by 
$$\til v=Ev\chi_{\Omega\cup(G\setminus F)}+\chi_{F\setminus\Omega}.$$  
Thus we have that $\til v\in\wtil K_{A,B,F,G}(\Omega).$  We note that $\Ha(\partial\Omega)<\infty$ implies that $P(\Omega,X)<\infty.$  Therefore by the Leibniz rule for BV functions \eqref{eq:Leibniz}, we have  
\begin{align*}
\|D\til v\|(X)&\le\|DEv\chi_{\Omega\cup(G\setminus F)}\|(X)+P(F\setminus\Omega,X)\\
	&\le \|DEv\|(X)+P(\Omega, X)+P(F,X)+P(G,X)<\infty,
\end{align*}
and so it follows that $\til v\in BV(X).$    

We claim that $\|D\til v\|(\partial\Omega)=0.$  Indeed, we have that 
$$\|DEv\|(\partial\Omega\cap(G\setminus\overline F))=0,$$ and so for all $\eta>0$, there exists an open set $U_\eta\supset\partial\Omega\cap G\setminus\overline F$ such that $$\|D\til v\|(U_\eta\cap(G\setminus\overline F))=\|DEv\|(U_\eta\cap (G\setminus\overline F))\le\|DEv\|(U_\eta)<\eta.$$  This follows since $\til v= Ev$ on the open set $U_\eta\cap(G\setminus\overline F)$.  Thus we have that $$\|D\til v\|(\partial\Omega\cap(G\setminus\overline F))=0.$$
Furthermore, since $T\til v=1$ $\Ha$-a.e.\ on $\partial\Omega\cap F,$ it follows that $\til v^\wedge=\til v^\vee$ $\Ha$-a.e.\ on $\partial\Omega\cap F.$  Thus the jump set $S_{\til v}$ of $\til v$ satisfies $\Ha(S_{\til v}\cap\partial\Omega\cap F)=0,$ and so $\|D\til v\|(\partial\Omega\cap F)=0.$  This follows from the decomposition \eqref{eq:Decomp} of the Radon measure $\|D\til v\|(\cdot)$ and the fact that $\Ha(\partial\Omega)<\infty.$  Similarly, since $T\til v=0$ on $\partial\Omega\setminus\overline G,$ we have that $\Ha(S_{\til v}\cap\partial\Omega\setminus\overline G)=0,$ and so $\|D\til v\|(\partial\Omega\setminus\overline G)=0.$  Since $\Ha(\partial\Omega\cap(\partial G\cup\partial F))=0,$ it follows that $\|D\til v\|(\partial\Omega)=0.$   

Therefore we have that 
\begin{align*}
P(E,\Omega)>\|Dv\|(\Omega)=\|D\til v\|(\Omega)&=\|D\til v\|(\overline\Omega)\\
	&=\inf\{\|D\til v\|(U):U\text{ open, }\overline\Omega\subset U\}.
\end{align*}
Hence, there exists an open set $U\supset\overline\Omega$ and $J\in\N$ such that for all $j\in\N$ with $j>J,$ we have that 
\begin{align*}
P(E,\Omega)>\|D\til v\|(U)\ge\|D\til v\|(\Omega_{1/k_j})\ge P(E_{k_j},\Omega_{1/k_j})=\lambda_{1/k_j}.
\end{align*}
The last inequality is due to the fact that $\til v\in\wtil\K_{A,B,F,G}(\Omega)$ and $\chi_{E_{k_j}}$ is a $1/k_j$-weak solution to the $\K_{A,B,F,G}$-obstacle problem.  Since this holds for all sufficiently large $j,$ we have that 
$$P(E,\Omega)>\liminf_{j\to\infty}\lambda_{1/k_j},$$ which contradicts \eqref{*}.  Thus we have that $P(E,\Omega)\le\|Dv\|(\Omega),$ and so $u:=\chi_E$ is a strong solution to the $\K_{A,B,F,G}$-obstacle problem.  
\end{proof}

\begin{remark}\label{rem:Extension}
In the previous proposition, it is possible to obtain the same result by replacing the uniform domain assumption with the assumption that $\Omega$ is a BV extension domain where extensions do not accumulate BV-energy on the boundary.  That is, if $Eu\in BV(X)$ is the extension of $u\in BV(\Omega),$ then $\|DEu\|(\partial\Omega)=0.$  It was shown in \cite{L1} that uniform domains have this property, and so for simplicity we assume that $\Omega$ is a uniform domain. 
\end{remark}

The following lemma is an analogue of Lemma~\ref{lem:WeakSolnTrace} for strong solution sets.

\begin{lem}\label{lem:StrongSolnTrace}
Assume the hypotheses of Proposition~\ref{prop:StrongSolnSet}.  Then for $\Ha$-a.e.\ $z\in\partial\Omega\cap F,$ there exists $\rho_z>0$ such that for any strong solution set $E$ to the $\K_{A,B,F,G}$-obstacle problem, we have that 
$$\mu(B(z,\rho_z)\cap\Omega\setminus E)=0.$$
Similarly, for $\Ha$-a.e.\ $z\in\partial\Omega\setminus\overline G,$ there exists $\rho_z>0$ such that for any strong solution set $E$ to the $\K_{A,B,F,G}$-obstacle problem, we have that 
$$\mu(B(z,\rho_z)\cap\Omega\cap E)=0.$$ 
\end{lem}

\begin{proof}
For $z\in\partial\Omega\cap F\cap B,$ there exists $r>0$ such that $B(z,r)\subset F$ and $B(z,r)\cap\overline\Omega\subset B.$  By Lemma~\ref{lem:SolnSetBound}, there exists $r_0>0$ such that $E_0\subset B(z,r),$ where $E_0$ is the regularized minimal weak solution to the Dirichlet problem with boundary data $\chi_{B(z,r_0)}.$  By positive mean curvature, there exists $\rho_z>0$ such that $B(z,\rho_z)\sqsubset E_0.$

By \cite[Proposition~4.8]{LMSS}, $E_0$ is also the minimal strong solution set, and so we have that $T\chi_{E_0}=\chi_{B(z,r_0)}$ $\Ha$-a.e.\ on $\partial\Omega.$  As $\chi_E\in\K_{A,B,F,G}(\Omega)$, we have that $T\chi_E(y)=1=T\chi_{E_0}(y)$ for $\Ha$-a.e.\ $y\in \partial\Omega\cap B(z,r_0)$.  Since $T\chi_{E_0}(y)=0$ for $\Ha$-a.e.\ $y\in\partial\Omega\setminus B(z,r_0),$ we see that $T\chi_{E\cap E_0}=\chi_{B(z,r_0)}$ $\Ha$-a.e.\ on $\partial\Omega.$  Likewise by our choice of $r_0,$ we have that $\chi_F\le T\chi_{E\cup E_0}\le \chi_G$ $\Ha$-a.e.\ on $\partial\Omega,$ and $\chi_{E\cup E_0}\in\K_{A,B,F,G}(\Omega).$  

By \cite[Proposition~4.7]{M}, we have that 
$$P(E\cap E_0,\Omega)+P(E\cup E_0,\Omega)\le P(E_0,\Omega)+P(E,\Omega).$$ 
If $P(E\cap E_0,\Omega)>P(E_0,\Omega),$ then $P(E\cup E_0,\Omega)<P(E,\Omega),$ which is a contradiction since $\chi_{E\cup E_0}\in\K_{A,B,F,G}(\Omega)$ and $E$ is a strong solution set to the $\K_{A,B,F,G}$-obstacle problem.  Hence $P(E\cap E_0,\Omega)\le P(E_0,\Omega),$ and so $E\cap E_0$ is a strong solution to the Dirichlet problem with boundary data $\chi_{B(z,r_0)}.$  By minimality of the strong solution set $E_0,$ we have that $E_0\cap\Omega\sqsubset E,$ and since $B(z,\rho_z)\sqsubset E_0,$ it follows that 
$$\mu(B(z,\rho_z)\cap\Omega\setminus E)=0.$$

For $\Ha$-a.e.\ $z\in\partial\Omega\setminus\overline G,$ we proceed as we did in Lemma~\ref{lem:WeakSolnTrace}, using positive mean curvature and Lemma~\ref{lem:SolnSetBound} to choose $r_0>0$ and $\rho_z>0$ such that 
$$B(z,\rho_z)\sqsubset E_0\subset B(z,r)\subset X\setminus(\overline G\cup A).$$  
Again, $E_0$ is the regularized minimal weak (and thus strong) solution to the Dirichlet problem with boundary data $\chi_{B(z,r_0)}.$ 
By our choice of $r_0,$ we see that $T\chi_{E_0\setminus E}=\chi_{B(z,r_0)}$ $\Ha$-a.e., and likewise that $\chi_F\le T\chi_{E\setminus E_0}\le\chi_G$ $\Ha$-a.e.\ on $\partial\Omega.$  We also have that $\chi_{E\setminus E_0}\in\K_{A,B,F,G}(\Omega).$  As in Lemma~\ref{lem:WeakSolnTrace}, we have that 
$$P(E_0\setminus E,\Omega)+P(E\setminus E_0,\Omega)\le P(E_0,\Omega)+P(E,\Omega),$$ and by the same argument as above, minimality of the strong solution $E_0$ implies that $E_0\sqsubset E_0\setminus E.$  Since $B(z,\rho_z)\sqsubset E_0,$ it follows that 
\[
\mu(B(z,\rho_z)\cap\Omega\cap E)=0.\qedhere
\]    
\end{proof}

By a similar argument as in the previous lemma, we can show that both the intersection and union of strong solution sets are strong solution sets.

\begin{lem}\label{lem:GenIntUnion}  Let $\Omega$ be a bounded domain with boundary of positive mean curvature.
For $i\in\{1,2\},$ let $A_i,B_i\subset\overline\Omega$ be relatively open, let $F_i\subset G_i\subset X$ be open, and suppose that 
\begin{enumerate}
\item[(i)] $P(F_i,X),P(G_i,X)<\infty,$
\item[(ii)] $\Ha(\partial\Omega\cap(\partial_{\overline\Omega} A_i\cup\partial F_i\cup\partial G_i))=0,$
\item[(iii)] $\chi_{F_i}\le\chi_{B_i}$ and $\chi_{A_i}\le\chi_{G_i}$ $\Ha$-a.e.\ on $\partial\Omega.$
\end{enumerate} 
In addition, suppose that $A_1\sqsubset A_2,$ $B_1\sqsubset B_2,$ $F_1\sqsubset F_2$, and $G_1\sqsubset G_2.$  If $E_1$ and $E_2$ are strong solution sets to the $\K_1:=\K_{A_1,B_1,F_1,G_1}$ and $\K_2:=\K_{A_2,B_2,F_2,G_2}$-obstacle problems respectively, then $E_1\cap E_2$ and $E_1\cup E_2$ are strong solution sets to the $\K_1$ and $\K_2$-obstacle problems respectively.  
\end{lem}

\begin{proof}
From \cite[Proposition~4.7]{M}, we have that 
\begin{equation}\label{eq:GenIntUnion}
P(E_1\cap E_2,\Omega)+P(E_1\cup E_2,\Omega)\le P(E_1,\Omega)+P(E_2,\Omega)<\infty,
\end{equation}
and so $\chi_{E_1\cap E_2},\chi_{E_1\cup E_2}\in BV(\Omega).$  Since $\chi_{A_1}\le\chi_{E_1}$ and $\chi_{A_1}\le\chi_{A_2}\le\chi_{E_2}$ $\mu$-a.e.\ in $\Omega,$ we have that $\chi_{A_1}\le\chi_{E_1\cap E_2}$ $\mu$-a.e.\ in $\Omega.$  Likewise we have that  $\chi_{E_1\cap E_2}\le\chi_{E_1}\le\chi_{B_1}$ $\mu$-a.e.\ in $\Omega,$ and so  
$$\chi_{A_1}\le\chi_{E_1\cap E_2}\le\chi_{B_1}$$
$\mu$-a.e.\ in $\Omega.$  Similarly, $\chi_{B_2}\le\chi_{E_2}\le\chi_{E_1\cup E_2}$ $\mu$-a.e.\ in $\Omega,$ and since $\chi_{E_1}\le\chi_{B_1}\le\chi_{B_2}$ $\mu$-a.e.\ in $\Omega$, it follows that 
$$\chi_{A_2}\le\chi_{E_1\cup E_2}\le\chi_{B_2}$$ 
$\mu$-a.e.\ in $\Omega.$  For $\Ha$-a.e.\ $z\in\partial\Omega\cap F_1,$ it follows from Lemma~\ref{lem:StrongSolnTrace} that there exists $\rho_{z,1}>0$ such that $B(z,\rho_{z,1})\cap\Omega\sqsubset E_1.$  Likewise, for $\Ha$-a.e.\ $z\in\partial\Omega\cap F_2,$ there exists $\rho_{z,2}>0$ such that $B(z,\rho_{z,2})\cap\Omega\sqsubset E_2.$  Since $F_1\sqsubset F_2,$ for $\Ha$-a.e.\ $z\in\partial\Omega\cap F_1$ there exists $\rho_z=\min\{\rho_{z,1},\rho_{z,2}\}>0$ such that $$B(z,\rho_z)\cap\Omega\sqsubset E_1\cap E_2.$$    
Thus $T_{E_1\cap E_2}(z)=1=\chi_{F_1}(z)$ for $\Ha$-a.e.\ $z\in\partial\Omega\cap F_1.$  Similarly, it follows from Lemma~\ref{lem:StrongSolnTrace} that for $\Ha$-a.e.\ $z\in\partial\Omega\setminus\overline G_1,$ there exists $\rho_{z}>0$ such that 
$$\mu(B(z,\rho_z)\cap\Omega\cap E_1\cap E_2)=\mu(B(z,\rho_z)\cap\Omega\cap E_1)=0.$$  Thus $T\chi_{E_1\cap E_2}(z)=0=\chi_{G_1}(z)$ for $\Ha$-a.e.\ $z\in\partial\Omega\cap\overline G_1.$  Since $\Ha(\partial\Omega\cap\partial G_1)=0$, it follows that 
$$\chi_{F_1}\le T\chi_{E_1\cap E_2}\le \chi_{G_1}$$ 
$\Ha$-a.e.\ on $\partial\Omega,$ and so $\chi_{E_1\cap E_2}\in\K_1.$  Using Lemma~\ref{lem:StrongSolnTrace} in a similar fashion, we have that 
$$\chi_{F_2}\le T\chi_{E_1\cup E_2}\le\chi_{G_2}$$ 
$\Ha$-a.e.\ on $\partial\Omega.$ Thus it follows that $\chi_{E_1\cup E_2}\in\K_2.$

Now if $P(E_1\cap E_2,\Omega)>P(E_1,\Omega),$ then $P(E_1\cup E_2,\Omega)<P(E_2,\Omega)$ by \eqref{eq:GenIntUnion} .  However this is a contradiction since $\chi_{E_1\cup E_2}\in\K_2$ and $E_2$ is a strong solution set to the $\K_2$-obstacle problem.  Therefore $P(E_1\cap E_2,\Omega)\le P(E_1,\Omega),$ and since $\chi_{E_1\cap E_2}\in\K_1,$ we have that $E_1\cap E_2$ is a strong solution set to the $\K_1$-obstacle problem.  By a similar argument, $E_1\cup E_2$ is a strong solution set to the $\K_2$-obstacle problem.      
\end{proof}

We conclude this section by showing that minimal strong solution sets exist.  This is an adaptation of \cite[Proposition~3.7]{LMSS} which gave a similar result for weak solutions to the Dirichlet problem with boundary data $\chi_F$ for $F\subset X$ open.

\begin{lem}\label{lem:MinStrongSolnSet}
Assume the hypotheses of Proposition~\ref{prop:StrongSolnSet}.  Then a unique minimal strong solution set $E$ exists to the $\K_{A,B,F,G}$-obstacle problem.
\end{lem}

\begin{proof}
Let $$\Ee:=\{E\subset\Omega:\chi_E\text{ is a strong solution to the $\K_{A,B,F,G}$-obstacle problem}\},$$ and let $\alpha:=\inf_{E\in\Ee}\mu(E).$  Since $\Ee\ne\varnothing$ by Proposition~\ref{prop:StrongSolnSet}, we have that $\alpha<\infty.$
Let $\{E_k\}_{k\in\N}\subset\Ee$ be such that $\mu(E_k)\to\alpha$ as $k\to\infty.$  Let $\wtil E_1=E_1$ and for $k\in\N,$ let 
$$\wtil E_{k+1}=\wtil E_k\cap E_{k+1}.$$  
By Lemma~\ref{lem:GenIntUnion}, we have that $\wtil E_k\in\Ee$ for all $k\in\N.$  

Letting $E=\bigcap_{k\in\N} \wtil E_k,$ it then follows that 
$$\mu(\wtil E_k)\to\alpha=\mu(E)$$ as $k\to\infty.$  We note that $\chi_A\le\chi_E\le\chi_B$ $\mu$-a.e.\ in $\Omega$ since the $\wtil E_k$ are strong solution sets, and likewise by Lemma~\ref{lem:StrongSolnTrace} it follows that $\chi_F\le T\chi_E\le\chi_G$ $\Ha$-a.e.\ on $\partial\Omega.$  

Since $\chi_{\wtil E_k}\to\chi_E$ in $L^1(\Omega),$ it follows from lower semicontinuity of the BV energy that 
$$P(E,\Omega)\le\liminf_{k\to\infty}P(\wtil E_k,\Omega)<\infty.$$
Hence $\chi_E\in BV(\Omega)$ and so $\chi_E\in\K_{A,B,F,G}(\Omega).$  Since the $\wtil E_k$ are strong solution sets, it follows that $E$ is a strong solution set.

Let $E'$ be another strong solution set to the $\K_{A,B,F,G}$-obstacle problem.  By Lemma~\ref{lem:GenIntUnion}, $E\cap E'$ is a strong solution set, and so it follows that 
$$\alpha\le\mu(E\cap E')\le\mu(E)=\alpha.$$  
Hence $\mu(E\setminus E')=0,$  and so $E$ is the minimal strong solution set to the $\K_{A,B,F,G}$-obstacle problem.  Uniqueness $\mu$-a.e.\ follows from minimality.
\end{proof}

\section{Continuous boundary data and continuous obstacles}

In this section, we assume that $\Omega$ is a bounded uniform domain such that $\mu(X\setminus\Omega)>0$, with boundary of positive mean curvature, such that $\Ha(\partial\Omega)<\infty$, $\Ha|_{\partial\Omega}$ is doubling, and $\Ha|_{\partial\Omega}$ is lower codimension 1 Ahlfors regular as in \eqref{eq:LowerAR}.  Furthermore, we assume that $\Ha(\{z\})=0$ for all $z\in\partial\Omega.$  

Our goal is to use the minimal strong solution sets given by Lemma~\ref{lem:MinStrongSolnSet} to construct strong solutions, in the manner of \cite{SWZ}, for continuous $\psi_1,$ $\psi_2,$ $f,$ and $g.$  We begin with the following simple lemma.

\begin{lem}\label{lem:NonEmpty}
Let $v\in\K_{\psi_1,\psi_2,f,g}(\Omega).$  Then for $\Leb$-a.e.\ $t\in\R,$ we have that $$\chi_{\{v>t\}}\in\K_{A_t,B_t,F_t,G_t}(\Omega),$$
where $A_t:=\{\psi_1>t\},$ $B_t:=\{\psi_2>t\},$ $F_t:=\{f>t\},$ and $G_t:=\{g>t\}.$
\end{lem}

\begin{proof}
Since $v\in BV(\Omega),$ it follows from the coarea formula that $\chi_{\{v>t\}}\in BV(\Omega)$ for $\Leb$-a.e.\ $t\in\R.$  If $z\in\Omega$ such that $z\in A_t\setminus\{v>t\}$ for $t\in\R,$ then $v(z)<\psi_1(z).$  Since $v\in\K_{\psi_1,\psi_2,f,g}(\Omega),$ this can happen for at most a $\mu$-negligible set of points in $\Omega.$  Thus $\chi_{A_t}\le\chi_{\{v>t\}}$ $\mu$-a.e.\ in $\Omega.$  Similarly, we have that $\chi_{\{v>t\}}\le\chi_{B_t}$ $\mu$-a.e.\ in $\Omega.$ 

By Lemma~\ref{lem:SupLevelTrace}, we have that $T\chi_{\{v>t\}}=\chi_{\{Tv>t\}}$ $\Ha$-a.e.\ on $\partial\Omega$ for $\Leb$-a.e.\ $t\in\R.$ Since $f\le Tv\le g$ $\Ha$-a.e.\ on $\partial\Omega,$ it follows that $\chi_{F_t}\le T\chi_{\{v>t\}}\le\chi_{G_t}$
$\Ha$-a.e.\ on $\partial\Omega$ for such $t\in\R.$   
\end{proof}

We now prove Theorem~\ref{thm:ExistStrongSoln} by adapting the proof of \cite[Theorem~4.10]{LMSS}.

\begin{proof}[Proof of Theorem~\ref{thm:ExistStrongSoln}.]
Since $\K_{\psi_1,\psi_2,f,g}(\Omega)\ne\varnothing$ and $\Ha$ satisfies \eqref{eq:LowerAR}, it follows that $f\le g$ on $\partial\Omega$ by continuity of $f$ and $g$.  Let $\Ext f,\Ext g\in C(X)\cap BV(X)$ be the extensions of $f$ and $g$ to $X$ constructed in \cite[Section~5]{LMSS}.  Replacing $\Ext f$ with $\min\{\Ext f,\Ext g\}$ if necessary, we may assume that $\Ext f\le \Ext g$ in $X$. For each $t\in\R$, let 
$$A_t:=\{\psi_1>t\},\, B_t:=\{\psi_2>t\},\, F_t:=\{\Ext f>t\},\, G_t:=\{\Ext g>t\}.$$ 

Since there exists $w\in\K_{\psi_1,\psi_2,f,g}(\Omega),$ it follows that $\psi_1\le g$ and $f\le\psi_2$ $\Ha$-a.e.\ on $\partial\Omega.$  Indeed if $z\in\partial\Omega$ is such that $g(z)<\psi_1(z),$ then by continuity of $\psi_1,$ there exists $r>0$ such that 
$$\psi_1>g(z)+(\psi_1(z)-g(z))/2$$
 on $B(z,r)\cap\Omega.$  Since $w\ge\psi_1$ $\mu$-a.e.\ in $\Omega,$ it follows that $Tw(z)>g(z).$  Since $w\in\K_{\psi_1,\psi_2,f,g}(\Omega),$ the set of $z\in\partial\Omega$ such that $Tw(z)>g(z)$ is $\Ha$-negligible.  Hence, $\psi_1\le g$ $\Ha$-a.e.\ on $\partial\Omega.$  A similar argument shows that $f\le\psi_2$ $\Ha$-a.e.\ on $\partial\Omega.$  Furthermore, non-emptiness of $\K_{\psi_1,\psi_2,f,g}(\Omega)$ and continuity of $\psi_1$ and $\psi_2$ implies that $\psi_1\le\psi_2$ in $\Omega$.

From these facts and continuity of the obstacle and boundary functions, it follows that for each $t\in\R,$ $A_t\subset B_t\subset\overline\Omega$ are relatively open, and $F_t\subset G_t\subset X$ are open such that $\chi_{A_t}\le\chi_{G_t}$ and $\chi_{F_t}\le\chi_{B_t}$ $\Ha$-a.e.\ on $\partial\Omega.$  Since $\Ext f, \Ext g\in BV(X),$ it follows from the coarea formula that $P(F_t,X),P(G_t,X)<\infty$ for $\Leb$-a.e.\ $t\in\R.$  Moreover from continuity of $\psi_1,f$ and $g$, and the fact that $\Ha(\partial\Omega)<\infty,$ it follows that $$\Ha(\partial\Omega\cap(\partial_{\overline\Omega} A_t\cup\partial F_t\cup\partial G_t))=0$$ for $\Leb$-a.e.\ $t\in\R.$  

Since there exists $w\in\K_{\psi_1,\psi_2,f,g}(\Omega),$ it follows that $w_t:=\chi_{\{w>t\}}\in\K_{A_t,B_t,F_t,G_t}(\Omega)$ for $\Leb$-a.e.\ $t\in\R$ by Lemma~\ref{lem:NonEmpty}.  Modifying $w_t$ on a set of measure zero if necessary, we may assume that $\chi_{A_t}\le w_t\le\chi_{B_t}$ everywhere in $\Omega.$  Since $w_t\in BV(\Omega),$ there exists by \cite[Theorem~3.1]{L1} an extension $Ew_t\in BV(X)$ such that $Ew_t|_\Omega=w_t$, $\|DEw_t\|(\partial\Omega)=0,$ and $0\le Ew_t\le 1.$ Let $\til w_t:X\to\R$ be given by
$$\til w_t= Ew_t\chi_{\Omega\cup(G_t\setminus F_t)}+\chi_{F_t\setminus\Omega}.$$  By using the same argument as in the proof of Proposition~\ref{prop:StrongSolnSet}, it follows that $\til w_t\in~\wtil\K_{A_t,B_t,F_t,G_t}(\Omega)\cap BV_\loc(X).$ Hence, we have shown that $A_t,B_t,F_t$ and $G_t$ satisfy the hypotheses of Proposition~\ref{prop:StrongSolnSet} for $\Leb$-a.e.\ $t\in\R.$  

Let $J\subset\R$ be the set of all $t\in\R$ satisfying these conditions.  Then for each $t\in J,$ there exists a minimal strong solution set $E_t$ to the $\K_{A_t,B_t,F_t,G_t}$-obstacle problem by Proposition~\ref{prop:StrongSolnSet} and Lemma~\ref{lem:MinStrongSolnSet}.   

For $s,t\in J$ with $s<t,$ it follows from Lemma~\ref{lem:GenIntUnion} that $E_s\cap E_t$ is a strong solution set to the $\K_{A_t,B_t,F_t,G_t}$-obstacle problem.  However since $E_t$ is the minimal strong solution set, it follows that $E_t\sqsubset E_s\cap E_t.$  Therefore, for $s,t\in J$ such that $s<t,$ we have 
\begin{equation}\label{eq:Nest}
E_t\sqsubset E_s.
\end{equation}

 Since $\Leb(\R\setminus J)=0,$ there exists a countable set $I\subset J$ which is dense in $\R.$  Let $u:\Omega\to\R$ be given by 
$$u(x)=\sup\{t\in I:x\in E_t\}.$$    
We will show that $u$ is a strong solution to the $\K_{\psi_1,\psi_2,f,g}$-obstacle problem.

First, let $x\in\Omega$ be such that $\chi_{A_t}(x)\le\chi_{E_t}(x)\le\chi_{B_t}(x)$ for all $t\in I.$  Since the $E_t$ are strong solution sets to the $\K_{A_t,B_t,F_t,G_t}$-obstacle problem, this property holds for $\mu$-a.e.\ $x\in\Omega.$  If $u(x)<\psi_1(x),$ then by density of $I$ in $\R$, there exists $t\in I$ such that $u(x)<t<\psi_1(x).$  This implies that $x\in A_t\setminus E_t,$ a contradiction.  Hence we have that $\psi_1(x)\le u(x).$  By a similar argument we have that $u(x)\le\psi_2(x),$ and so $\psi_1\le u\le\psi_2$ $\mu$-a.e.\ in $\Omega.$

To show that $u$ satisfies the correct trace, let $$z\in\partial\Omega\setminus\bigcup_{t\in I}(\partial F_t\cup\partial G_t)$$ be such that $f(z)\le g(z),$ and such that for all $t\in I,$ the conclusion of Lemma~\ref{lem:StrongSolnTrace} holds if $z\in\partial\Omega\cap F_t$ or $z\in\partial\Omega\setminus\overline G_t.$  By the hypotheses, $\Ha$-a.e.\ $z\in\partial\Omega$ satisfy these conditions.  

Let $t\in I$ such that $t<f(z).$  It follows that $z\in\partial\Omega\cap F_t,$ and so by Lemma~\ref{lem:StrongSolnTrace}, there exists $r_z>0$ such that $\mu(B(z,r_z)\cap\Omega\setminus E_t)=0.$  Thus $y\in E_t$ for $\mu$-a.e.\ $y\in B(z,r_z)\cap\Omega,$ which implies that $u(y)\ge t$ for such $y.$  Hence $Tu(z)\ge t,$ and since this holds for all $t\in I$ with $t<f(z),$ we have that $f(z)\le Tu(z).$       

Likewise, let $t\in I$ such that $g(z)<t.$  Then $z\in\partial\Omega\setminus\overline G_t,$ and so by Lemma~\ref{lem:StrongSolnTrace}, there exists $r_z>0$ such that $\mu(B(z,r_z)\cap\Omega\cap E_t)=0.$  Thus $y\not\in E_t$ for $\mu$-a.e.\ $y\in B(z,r_z)\cap\Omega,$ and so it follows that $u(y)\le t$ for all such $y$.  This implies that $Tu(z)\le t,$ and since this holds for all $t\in I$ with $g(z)<t,$ we have that $Tu(z)\le g(z).$  Therefore, it follows that $f\le Tu\le g$ $\Ha$-a.e.\ on $\partial\Omega.$ 

It remains to show that $u\in BV(\Omega)$ and that $u$ satisfies the energy minimization property.  To this end, let $v\in\K_{\psi_1,\psi_2,f,g}(\Omega).$  Such a $v$ exists since $\K_{\psi_1,\psi_2,f,g}(\Omega)\ne\varnothing$.  For $t\in J,$ it follows from the definition of $u$ and \eqref{eq:Nest} that 
$$\{u>t\}=\bigcup_{I\ni s>t}E_s\sqsubset E_t\subset\{u\ge t\}.$$  Since $\mu$ is Radon and doubling, hence  $\sigma$-finite on $X$, we have that $$\mu(\{u=t\})=0$$ for $\Leb$-a.e.\ $t\in J.$  For all such $t,$ we have that $E_t\sqsubset\{u>t\},$ hence $\chi_{E_t}=\chi_{\{u>t\}}$ $\mu$-a.e.\ in $\Omega.$  Thus $\{u>t\}$ is a strong solution set to the $\K_{A_t,B_t,F_t,G_t}$-obstacle problem for $\Leb$-a.e.\ $t\in\R.$  Since $v\in\K_{\psi_1,\psi_2,f,g}(\Omega),$ we have by Lemma~\ref{lem:NonEmpty} that $\chi_{\{v>t\}}\in\K_{A_t,B_t,F_t,G_t}(\Omega)$ for $\Leb$-a.e.\ $t\in\R.$  Hence by the coarea formula, it follows that 
\begin{align*}
\|Du\|(\Omega)=\int_\R P(\{u>t\},\Omega)dt\le\int_{\R}P(\{v>t\},\Omega)dt=\|Dv\|(\Omega)<\infty.
\end{align*}
Therefore as $u\in BV(\Omega)$, we have that $u\in\K_{\psi_1,\psi_2,f,g}(\Omega)$, and $u$ satisfies the energy minimization property. Thus $u$ is a strong solution to the $\K_{\psi_1,\psi_2,f,g}$-obstacle problem. 
\end{proof} 

In this proof, we have again used the fact that $\Omega$ is a uniform domain to apply the BV extension result \cite[Theorem~3.1]{L1}, as we did in the proof of Proposition~\ref{prop:StrongSolnSet}.  However, as pointed out in Remark~\ref{rem:Extension}, we could replace the uniform domain assumption with the assumption that $\Omega$ is a BV extension domain for which extensions do not accumulate BV-energy on the boundary.   

In the following manner, it is also the case that strong solutions to the $\K_{\psi_1,\psi_2,f,g}$-obstacle problem yield strong solutions to the problem where the obstacles and boundary data are superlevel sets of the original functions: 

\begin{lem}\label{lem:DObsSupLevelSoln}
Let $f,g\in C(\partial\Omega)$ and let $\psi_1,\psi_2\in C(\overline\Omega)$. Furthermore, suppose that $\K_{\psi_1,\psi_2,f,g}(\Omega)\ne\varnothing.$  If $v\in\K_{\psi_1,\psi_2,f,g}(\Omega)$ is a strong solution to the $\K_{\psi_1,\psi_2,f,g}$-obstacle problem, then $\chi_{\{v>t\}}$ is a strong solution to the $\K_{A_t,B_t,F_t,G_t}$-obstacle problem for $\Leb$-a.e.\ $t\in\R,$ where $A_t,B_t,F_t,$ and $G_t$ are as defined in the proof of Theorem~\ref{thm:ExistStrongSoln}.
\end{lem}

\begin{proof}
Let $u$ be the strong solution to the $\K_{\psi_1,\psi_2,f,g}$-obstacle problem constructed in the proof of Theorem~\ref{thm:ExistStrongSoln} above.  In that proof, we saw that $\chi_{\{u>t\}}$ is a strong solution to the $\K_{A_t,B_t,F_t,G_t}$-obstacle problem for $\Leb$-a.e.\ $t\in\R.$  Let $\wtil J\subset\R$ be the set of all such $t.$   

Let $v$ be a strong solution to the $\K_{\psi_1,\psi_2,f,g}$-obstacle problem.  Since $v\in\K_{\psi_1,\psi_2,f,g}(\Omega),$  we have $\chi_{A_t}\le\chi_{\{v>t\}}\le\chi_{B_t}$ $\mu$-a.e.\ in $\Omega$ for all $t\in\R.$  By the coarea formula, $P(\{v>t\},\Omega)<\infty$ for $\Leb$-a.e.\ $t\in\R$.  Furthermore, by Lemma~\ref{lem:SupLevelTrace}, we have that $\chi_{F_t}\le T\chi_{\{v>t\}}\le\chi_{G_t}$ $\Ha$-a.e.\ on $\partial\Omega$ for $\Leb$-a.e.\ $t\in\R$, and so 
$$\chi_{\{v>t\}}\in\K_{A_t,B_t,F_t,G_t}(\Omega)$$ for $\Leb$-a.e.\ $t\in\R.$  Let $\mathcal{J}\subset\R$ be the set of all such $t.$ 

 Suppose that there exists a set $K\subset\mathcal{J}$ with $\Leb(K)>0$ such that for all $t\in K,$ $\chi_{\{v>t\}}$ is not a strong solution to the $\K_{A_t,B_t,F_t,G_t}$-obstacle problem.  For each $t\in \wtil J\cap K,$ it then follows that $P(\{v>t\},\Omega)>P(\{u>t\},\Omega).$  By the coarea formula, we have that 
\begin{align*}\int_{(\mathcal{J}\cap\wtil J)\setminus K}P(\{u&>t\},\Omega)dt+\int_{\wtil J\cap K}P(\{u>t\},\Omega)dt=\|Du\|(\Omega)=\|Dv\|(\Omega)\\
	&=\int_{(\mathcal{J}\cap\wtil J)\setminus K}P(\{v>t\},\Omega)dt+\int_{\wtil J\cap K}P(\{v>t\},\Omega)dt\\
	&>\int_{(\mathcal{J}\cap\wtil J)\setminus K}P(\{v>t\},\Omega)dt+\int_{\wtil J\cap K}P(\{u>t\},\Omega)dt.\\
\end{align*}  
Hence it follows that $$\int_{(\mathcal{J}\cap\wtil J)\setminus K}P(\{u>t\},\Omega)dt>\int_{(\mathcal{J}\cap\wtil J)\setminus K}P(\{v>t\},\Omega)dt.$$
However for all $t\in(\mathcal{J}\cap\wtil J)\setminus K$, we have that $P(\{u>t\},\Omega)\le P(\{v>t\},\Omega)$ since $\chi_{\{u>t\}}$ is a strong solution to the $\K_{A_t,B_t,F_t,G_t}$-obstacle problem.  Thus $\Leb((\mathcal{J}\cap\wtil J)\setminus K)=0,$ and so we have that
$$\|Dv\|(\Omega)>\|Du\|(\Omega),$$ a contradiction.  Hence it follows that $\chi_{\{v>t\}}$ is a strong solution to the $\K_{A_t,B_t,F_t,G_t}$-obstacle problem for $\Leb$-a.e.\ $t\in\R.$   
\end{proof}

For $\psi_1,\psi_2,f,$ and $g$ as in Theorem~\ref{thm:ExistStrongSoln}, strong solutions to the $\K_{\psi_1,\psi_2,f,g}$-obstacle problem may fail to be unique or even continuous.  See \cite[Section~6]{LMSS} for examples in the weighted Euclidean disk illustrating this for the special case of the Dirichlet problem with (single) continuous boundary data.  As such, we do not have a comparison principle in general.  However, it was shown in \cite{K} that unique minimal strong solutions exist for the Dirichlet problem with continuous boundary data, resulting in a comparison principle for such minimal strong solutions.  We prove the analogous result for the $\K_{\psi_1,\psi_2,f,g}$-obstacle problem and likewise obtain a comparison principle in Theorem~\ref{cor:DObsCompThm}.   

\begin{prop}\label{prop:DObsExistMin}
Let $f,g\in C(\partial\Omega)$, let $\psi_1,\psi_2\in C(\overline\Omega)$, and suppose that $\K_{\psi_1,\psi_2,f,g}(\Omega)\ne\varnothing.$  Then the function $u$ constructed in the proof of Theorem~\ref{thm:ExistStrongSoln} is the unique minimal strong solution to the $\K_{\psi_1,\psi_2,f,g}$-obstacle problem.
\end{prop}

\begin{proof}
Let $I\subset J\subset\R$ be as in the proof of Theorem~\ref{thm:ExistStrongSoln}, and recall that $\Leb(\R\setminus J)=0$ with $I$ countable and dense in $\R.$  Let $$u(x)=\sup\{t\in I:x\in E_t\}$$ be the strong solution constructed in that proof and let $u'$ be another strong solution.  By Lemma~\ref{lem:DObsSupLevelSoln}, there exists a set $J'\subset J$ with $\Leb(J\setminus J')=0$ such that $\chi_{\{u'>t\}}$ is a strong solution to the $\K_{A_t,B_t,F_t,G_t}$-obstacle problem for all $t\in J'.$ Let $I'\subset J'\subset J$ be a countable set, dense in $\R.$  

For each $t'\in I',$ we have that $E_{t'}\sqsubset\{u'>t'\}$ since $E_{t'}$ is the minimal strong solution set to the $\K_{A_{t'},B_{t'},F_{t'},G_{t'}}$-obstacle problem.  Therefore, setting $$U:=\bigcup_{t'\in I'}E_{t'}\setminus\{u'>t'\},$$ we have that $\mu(U)=0.$  Likewise, for each $t\in I$ and $t'\in I'$ with $t'<t,$ we have that $E_t\sqsubset E_{t'}$ by \eqref{eq:Nest} and the fact that $I'\subset J.$  Setting $$V:=\bigcup_{t'\in I'}\bigcup_{t\in I,t'<t}E_t\setminus E_{t'},$$ it follows that $\mu(V)=0.$

Let $x\in\Omega\setminus{(U\cup V)}$, and suppose that $u'(x)<u(x).$  By the definition of $u$, there exists $t\in I$ such that $u'(x)<t\le u(x)$ and $x\in E_t.$  By the density of $I'$ in $\R$, there exists $t'\in I'$ such that 
$$u'(x)<t'<t\le u(x).$$  Since $x\not\in V,$ we have that $x\in E_{t'}\setminus\{u'>t'\}.$ This is a contradiction, as $x\not\in U,$ and so we have that $u(x)\le u'(x)$ for all $x\in\Omega\setminus(U\cup V).$  Thus $u$ is the minimal strong solution to the $\K_{\psi_1,\psi_2,f,g}$-obstacle problem.  Uniqueness $\mu$-a.e.\ follows from the minimality.  
\end{proof}

\begin{remark}\label{remark:MaxSoln}
For $f,g,\psi_1,\psi_2$ as above, we note that if $v\in\K_{\psi_1,\psi_2,f,g}(\Omega),$ then $-v\in\K_{-\psi_2,-\psi_1,-g,-f}(\Omega).$  Thus there exists a unique minimal strong solution $w$ to the $\K_{-\psi_2,-\psi_1,-g,-f}$-obstacle problem.  It then follows that $-w$ is the unique maximal strong solution to the $\K_{\psi_1,\psi_2,f,g}$-obstacle problem.  
\end{remark}

\begin{proof}[Proof of Theorem~\ref{cor:DObsCompThm}]
Existence of $u_1$ and $u_2$ follows from Proposition~\ref{prop:DObsExistMin}.  From the hypotheses, we have that $\psi_1\le\psi_2\le u_2$, $\psi_1\le u_1,$ and $\min\{u_1,u_2\}\le u_1\le\phii_1$ $\mu$-a.e.\ in $\Omega.$  Thus we see that $\psi_1\le\min\{u_1,u_2\}\le\phii_1$ $\mu$-a.e.\ in $\Omega.$  Similarly $\psi_2\le u_2\le\max\{u_1,u_2\}$, $u_1\le\phii_1\le\phii_2$, and $u_2\le\phii_2$ $\mu$-a.e.\ in $\Omega,$ and so it follows that $\psi_2\le\max\{u_1,u_2\}\le\phii_2$ $\mu$-a.e.\ in $\Omega.$   

Let $z\in\partial\Omega$ such that $f_1(z)\le f_2(z),$ $g_1(z)\le g_2(z)$, $f_1(z)\le Tu_1(z)\le g_1(z)$, and $f_2(z)\le Tu_2(z)\le g_2(z).$  These conditions hold for $\Ha$-a.e.\ $z\in\partial\Omega.$  For such $z$, it follows that 
\begin{align*}T\min\{u_1,u_2\}(z)&=\min\{Tu_1(z),Tu_2(z)\}\\
	T\max\{u_1,u_2\}(z)&=\max\{Tu_1(z),Tu_2(z)\},
\end{align*}
(see the proof of \cite[Lemma~4.4]{K}, for example).  Thus, $$f_1\le T\min\{u_1,u_2\}\le g_1\text{ and }f_2\le T\max\{u_1,u_2\}\le g_2$$ $\Ha$-a.e.\ on $\partial\Omega.$  By \cite[Lemma~3.1]{L2}, we have that 
\begin{align*}
\|D\min\{u_1,u_2\}\|(\Omega)+\|D\max\{u_1,u_2\}\|(\Omega)\le\|Du_1\|(\Omega)+\|Du_2\|(\Omega)<\infty,
\end{align*}
and so it follows that $\min\{u_1,u_2\}\in\K_1(\Omega)$ and $\max\{u_1,u_2\}\in\K_2(\Omega).$

If $\|D\min\{u_1,u_2\}\|(\Omega)>\|Du_1\|(\Omega),$ then $\|D\max\{u_1,u_2\}\|(\Omega)<\|Du_2\|(\Omega).$ This is a contradiction, since $u_2$ is a strong solution to the $\K_2$-obstacle problem and $\max\{u_1,u_2\}\in\K_2(\Omega),$ and so it follows that $$\|D\min\{u_1,u_2\}\|(\Omega)\le\|Du_1\|(\Omega).$$  In particular, $\min\{u_1,u_2\}$ is a strong solution to the $\K_1$-obstacle problem.  A similar argument shows that $\max\{u_1,u_2\}$ is a strong solution to the $\K_2$-obstacle problem.  Since $u_1$ is a minimal strong solution to the $\K_1$-obstacle problem, it follows that $$u_1\le\min\{u_1,u_2\}\le u_2$$ $\mu$-a.e.\ in $\Omega.$  

An analogous argument shows that $\max\{v_1,v_2\}$ is a strong solution to the $\K_2$-obstacle problem, and so maximality of $v_2$ implies that $$v_1\le\max\{v_1,v_2\}\le v_2$$ $\mu$-a.e.\ in $\Omega.$    
\end{proof}

\section{Stability with respect to single boundary condition}

In this section we compare the obstacle problem with double boundary condition to the standard problem when only one boundary condition is imposed.  We begin by providing straightforward examples, in the special case without obstacles, that for continuous boundary data $f\le g$, minimal and maximal strong solutions to the $\K_{f,g}$-problem are not in general strong solutions to the Dirichlet problem with boundary data $f$ or $g$.  Here we set $$\K_{f,g}(\Omega):=\{u\in BV(\Omega):f\le Tu\le g\;\Ha\text{-a.e.\ on } \partial\Omega\}.$$

\begin{example}
Let $\Omega\subset X$ be a bounded domain.  If $f,g\in C(\partial\Omega)$ are both not constant such that $$\sup_{\partial\Omega}f<\inf_{\partial\Omega} g,$$ then $u\equiv\sup_{\partial\Omega} f$ is the minimal strong solution to the $\K_{f,g}$-problem.  Likewise, $v\equiv\inf_{\partial\Omega}g$ is the maximal strong solution.  However, neither $u$ nor $v$ are strong solutions to the Dirichlet problems with boundary data $f$ or $g$.  
\end{example}

\begin{example}\label{ex:DiffProb}
Let $\Omega=\D\subset\R^2,$ and consider the continuous functions $f,g:\partial\Omega\to\R$ given by 
$$f(x,y)=y\qquad\text{and}\qquad g(x,y)=y+1.$$
From Proposition~\ref{prop:DObsExistMin}, the minimal strong solution to the $\K_{f,g}$-problem is given by $u(x)=\sup\{t\in I:x\in E_t\}$ where $E_t$ is the minimal strong solution to the $\K_{F_t,G_t}$-problem.  As in the previous section, we let $F_t:=\{f>t\}$ and $G_t:=\{g>t\}.$ 

Since $\Omega$ is the unit disk, it follows that for each $t\in\R$, the minimal strong solution set to the $\K_{F_t,G_t}$-obstacle problem is a subset of $\Omega$ which has boundary in $\Omega$ consisting of straight line segments.  Denoting $E_{F_t}$ and $E_{G_t}$ the minimal strong solutions to the Dirichlet problems with boundary data $\chi_{F_t}$ and $\chi_{G_t}$ respectively, we see that $E_{F_t}$ is the region of the disk bounded by the arc $F_t$ and the chord joining the endpoints of $F_t,$ and $E_{G_t}$ is the analogous region (see Figure~\ref{fig:G_t}). The minimal strong solution $E_t$ to the $\K_{F_t,G_t}$-obstacle problem can be chosen from the sets $E\subset\Omega$ for which $E_{F_t}\subset E\subset E_{G_t}$ and such that $\partial E\cap\Omega$ is a straight line segment lying in $E_{G_t}\setminus E_{F_t}.$ 

For $t<1/2,$ the set described above with the smallest perimeter, that is, whose boundary in $\Omega$ is comprised by the shortest line segment is the set $E_{G_t}.$  Likewise for $1/2\le t$, $E_{F_t}$ is the minimal strong solution set.  Hence, we have that 
$$
E_t=
\begin{cases} 
E_{G_t},&t<1/2\\
E_{F_t},&1/2\le t.
\end{cases} 
$$
Note that for $t<0,$ $G_t=\partial\Omega$ and so $E_{G_t}=\Omega.$  Likewise for $t>1,$ $F_t=\varnothing,$ and so $E_{F_t}=\varnothing.$ From the definition of $u$, it follows that $0\le u\le 1,$ and 
$$u(x,y)=
\begin{cases}
y+1,&-1<y<-1/2\\
1/2,&-1/2\le y\le 1/2\\
y,&1/2<y.
\end{cases}
$$
Hence, 
$$
Tu(x,y)=
\begin{cases}
g(x,y),&-1<y<-1/2\\
1/2,&-1/2\le y\le 1/2\\
f(x,y),&1/2<y,
\end{cases}
$$
and so $u$ is not a strong solution to the Dirichlet problem with either boundary data $f$ or $g,$ see Figure~\ref{fig:uGraph}.

\begin{figure}[h]
\centering
\begin{minipage}[t]{0.45\textwidth}
\centering
\definecolor{ffqqqq}{rgb}{1.,0.,0.}
\begin{tikzpicture}[line cap=round,line join=round,>=triangle 45,x=2.25cm,y=2.25cm]
\clip(-1.25,-1.25) rectangle (1.25,2.);
\draw [line width=0.4pt] (0.,0.) circle (2.25cm);
\draw [shift={(0.,0.)},line width=1.6pt,color=ffqqqq,fill=ffqqqq,fill opacity=0.1899999976158142]  plot[domain=-0.5346160946250951:3.6682853087080156,variable=\t]({1.*1.*cos(\t r)+0.*1.*sin(\t r)},{0.*1.*cos(\t r)+1.*1.*sin(\t r)});
\begin{small}
\draw[color=black] (0.012025260274884805,1.3110295133082293) node {$G_t$};
\draw[color=ffqqqq] (0.02307497270946536,0.3607542439343041) node {$E_{G_t}$};
\end{small}
\end{tikzpicture}
\caption{\small{The sets $G_t$ and $E_{G_t}$ when $t=1/2.$}}\label{fig:G_t}
\end{minipage}
\hfill
\begin{minipage}[t]{0.45\textwidth}
\centering
\includegraphics[scale=0.45]{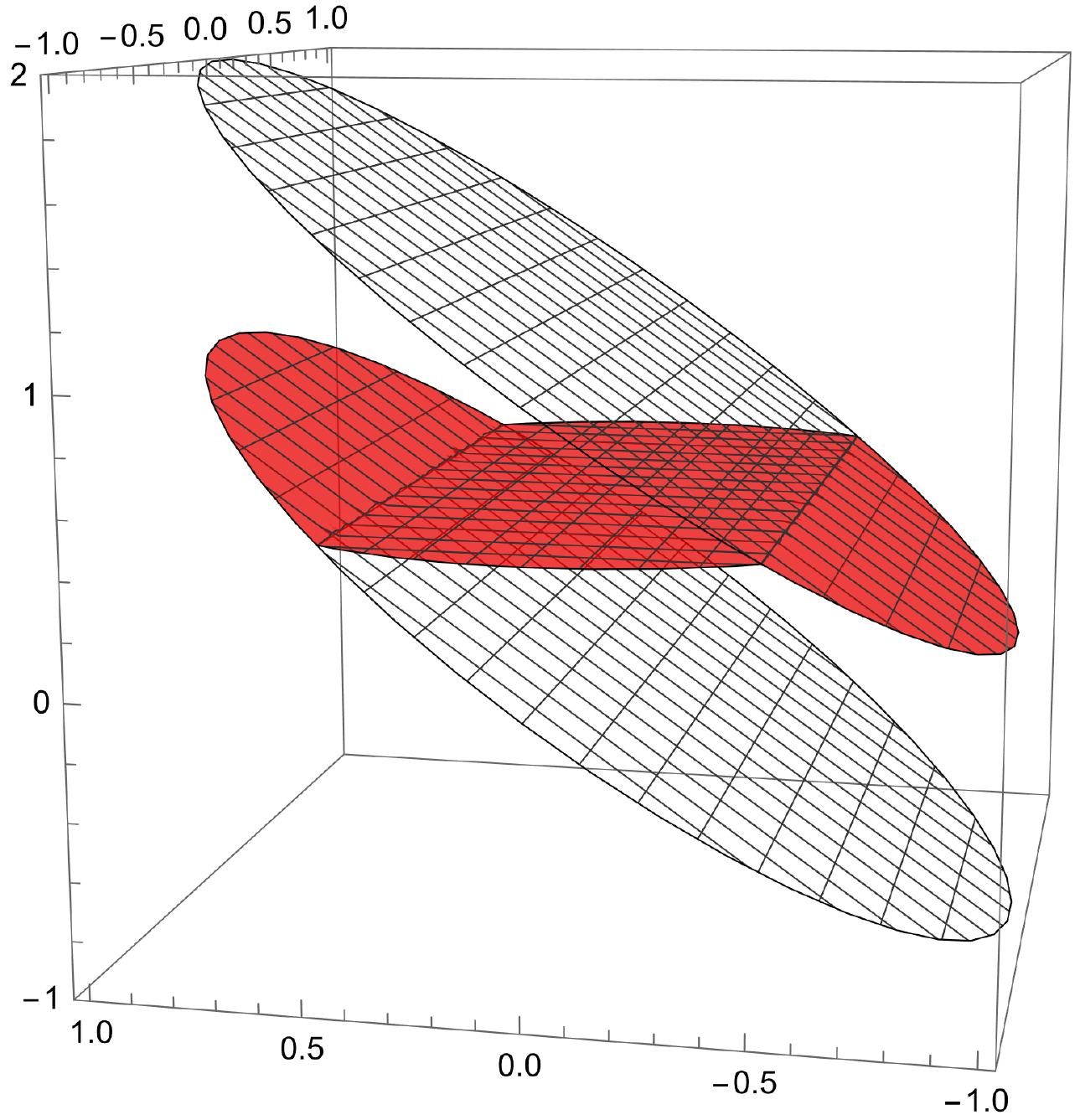}
\caption{\small{The shaded surface is the graph of $u$, whereas the upper and lower transparent surfaces are the graphs of the strong solutions to the Dirichlet problem with boundary data $g$ and $f$ respectively.}}\label{fig:uGraph}
\end{minipage}
\end{figure}

\end{example}

\begin{remark}
For $u,$ $f$, and $g$ as in Example~\ref{ex:DiffProb}, it follows that $u$ is the strong minimal solution to both the $\K_{Tu,g}$ and $\K_{f,Tu}$-problems, and $u$ is the minimal strong solution to the Dirichlet problem with boundary data $Tu.$  Thus, this example shows that all three scenarios are possible: for continuous boundary data $f$ and $g$ with $f\le g$, the minimal strong solution to the $\K_{f,g}$-problem may be the minimal strong solution to the Dirichlet problem with boundary data $f$ or $g.$   
\end{remark}

In Example~\ref{ex:DiffProb}, the solution $u$ fails to be a solution to the Dirichlet problem with boundary data $f$ or $g$ in part because the ranges of $f$ and $g$ overlap and the range of $f$ is not contained in the range of $g$.  In the metric setting, this proves to be a sufficient condition for the distinction of solutions to the $\K_{f,g}$-problem and Dirichlet problems, as shown in the following proposition.  

For the remainder of this section, assume that $\Omega$ is a bounded uniform domain with boundary of positive mean curvature such that $\Ha(\partial\Omega)<\infty$, $\Ha|_{\partial\Omega}$ is doubling, and that \eqref{eq:LowerAR} holds.  Furthermore, assume that $\Ha(\{z\})=0$ for all $z\in\partial\Omega.$

\begin{prop}\label{prop:DifferentProb}
Let $f,g\in C(\partial\Omega)$ be such that $f\le g$ $\Ha$-a.e.\ on $\partial\Omega,$ and let $u\in\K_{f,g}(\Omega)$ be the minimal strong solution to the $\K_{f,g}$-obstacle problem.  If  
$$\inf_{\partial\Omega} f<\inf_{\partial\Omega} g\le\sup_{\partial\Omega}f<\sup_{\partial\Omega} g,$$
then $\inf_{\partial\Omega}g\le u\le\sup_{\partial\Omega} f$ $\mu$-a.e.\ in $\Omega.$  In particular, $u$ is not a strong solution to the Dirichlet problem with boundary data $f$ or $g$.
\end{prop}

\begin{proof}
By uniqueness $\mu$-a.e.\ of the minimal strong solution $u$, we may assume that $u$ has the form given in the proof of Theorem~\ref{thm:ExistStrongSoln}.  That is, 
$$u(x)=\sup\{t\in I:x\in E_t\},$$ where
$I\subset J$ is a countable dense subset of $\R$, and $E_t$ is the minimal strong solution set to the $\K_{F_t,G_t}$-obstacle problem.  As in the proof of Theorem~\ref{thm:ExistStrongSoln}, $F_t=\{\Ext f>t\}$ and $G_t=\{\Ext g>t\}$, where $\Ext f,\Ext g\in C(X)\cap BV(X)$ are the extensions given in \cite[Section~5]{LMSS}.

For each $t\in I$ with $\inf_{\partial\Omega} f<t<\inf_{\partial\Omega} g,$ we have that $G_t=\partial\Omega$ and $\varnothing\neq F_t\subsetneq\partial\Omega.$  Thus $\mu(E_t)>0$ and since $\chi_\Omega\in\K_{F_t,G_t},$ it follows that $P(E_t,\Omega)=0.$  We then have that $E_t=\Omega$.  Indeed, $E_t$ must have full measure in $\Omega,$ or otherwise by connectedness of $\Omega,$ there would exist some ball $B\subset\Omega$ such that $\min\{\mu(B\cap E_t),\mu(B\setminus E_t)\}>0.$  However, by the relative isoperimetric inequality and the fact that $P(E_t,\Omega)=0,$ this is a contradiction, and so $E_t=\Omega.$  Since this holds for all $t\in I,$ it follows from the definition of $u$ that $\inf_{\partial\Omega}g\le u$ $\mu$-a.e.\ in $\Omega.$  Likewise for all $t\in I$ such that $\sup_{\partial\Omega}f<t,$ we have that $F_t=\varnothing,$ and so $E_t=\varnothing.$  Thus $u\le t,$ and since this holds for all such $t$, it follows that $u\le\sup_{\partial\Omega}f.$

Since $f$ and $g$ are continuous, and the outer inequalities in the statement of the proposition are strict, there are subsets of $\partial\Omega$ with $\Ha$-positive measure where $Tu\ne f$ and where $Tu\ne g.$  Thus $u$ cannot be a strong solution to either the Dirichlet problem with boundary data $f$ or $g$.  
\end{proof}

By replacing the function $g$ in Example~\ref{ex:DiffProb} with $g_\eps(x,y)=y+\eps,$ one can see that the minimal strong solution $u_\eps$ to the $\K_{f,g_\eps}$-problem converges to the strong solution the Dirichlet problem with boundary data $f$ as $\eps\to 0.$  Likewise, if $f_\eps(x,y)=y+\eps$ for $0<\eps<1$ and $u_\eps$ is the minimal strong solution to the $\K_{f_\eps,g}$-problem, then $u_\eps$ converges to the strong solution to the Dirichlet problem with boundary data $g$ as $\eps\to 1.$  This stability also holds in the metric setting when obstacles are involved.   

\begin{thm}\label{thm:Stability} 
Let $\psi_1,\psi_2\in C(\overline\Omega)$ and for all $j,k\in\N$, let $f,g_j,h_k\in C(\partial\Omega)$ be such that 
 $g_j,h_k\to f$ as $j,k\to\infty$ pointwise $\Ha$-a.e.\ on $\partial\Omega,$ and such that 
$$g_j\le g_{j+1}\le f\le h_{k+1}\le h_k$$ $\Ha$-a.e.\ on $\partial\Omega$.  Furthermore, suppose the following are all non-empty: \\$\K_{\psi_1,\psi_2,f,f}(\Omega),$ $\K_{\psi_1,\psi_2,g_j,f}(\Omega),$ $\K_{\psi_1,\psi_2,f,h_k}(\Omega),$ and $\K_{\psi_1,\psi_2,g_j,h_k}(\Omega),$ for all \\$j,k\in\N.$ 

\begin{enumerate}
\item[(a)] If $u_j$ is the minimal strong solution to the $\K_{\psi_1,\psi_2,g_j,f}$-obstacle problem for each $j\in\N$, then $u:=\lim_j u_j$ is the minimal strong solution to the $\K_{\psi_1,\psi_2,f,f}$-obstacle problem, and $u_j\to u$ in $L^1(\Omega).$\\
\item[(b)] If $u_k$ is the maximal strong solution to the $\K_{\psi_1,\psi_2,f,h_k}$-obstacle problem for each $k\in\N$, then $u:=\lim_k u_k$ is the maximal strong solution to the $\K_{\psi_1,\psi_2,f,f}$-obstacle problem, and $u_k\to u$ in $L^1(\Omega).$\\
\item[(c)]If $u_{j,k}$ is the minimal strong solution to the $\K_{\psi_1,\psi_2,g_j,h_k}$-obstacle problem, then 
$$u:=\lim_{k\to\infty}\lim_{j\to\infty}u_{j,k}$$ is a strong solution to the $\K_{\psi_1,\psi_2,f,f}$-obstacle problem.\\
\item[(d)]If $u_{j,k}$ is the maximal strong solution to the $\K_{\psi_1,\psi_2,g_j,h_k}$-obstacle problem, then 
$$u:=\lim_{j\to\infty}\lim_{k\to\infty}u_{j,k}$$ is a strong solution to the $\K_{\psi_1,\psi_2,f,f}$-obstacle problem.
\end{enumerate} 
\end{thm}

The proofs for (b) and (d) are analogous to those for (a) and (c) respectively, so we only prove (a) and (c) here.  Notice that in parts (c) and (d), this result does not give us the minimal or maximal strong solution, but only some strong solution.

\begin{proof}
(a)  Let $u_f$ be the minimal strong solution to the $\K_{\psi_1,\psi_2,f,f}$-obstacle problem, which exists by Theorem~\ref{thm:ExistStrongSoln}.  Since $g_j\le g_{j+1}\le f$ $\Ha$-a.e.\, it follows from Theorem~\ref{cor:DObsCompThm} that 
\begin{equation}\label{eqn:u_f}u_j\le u_{j+1}\le \lim_{j\to\infty}u_j=u\le u_f\end{equation} $\mu$-a.e.\ in $\Omega.$  Since $u_f,u_j\in L^1(\Omega)$ for each $j\in\N$ (as they are in $BV(\Omega)$), it follows that $u_j\to u$ in $L^1(\Omega).$  By lower-semicontinuity of the BV energy, we then have that
$$\|Du\|(\Omega)\le\liminf\|Du_j\|(\Omega).$$  For each $j\in\N,$ we have that $u_f\in\K_{\psi_1,\psi_2,f,f}(\Omega)\subset\K_{\psi_1,\psi_2,g_j,f}(\Omega),$ and so it follows that $\|Du_j\|(\Omega)\le\|Du_f\|(\Omega)$.  Therefore, we have that 
\begin{equation}\label{eqn:u_fEnergy}\|Du\|(\Omega)\le\|Du_f\|(\Omega)<\infty.\end{equation} 
Since $u=\lim_{j\to\infty}u_j$ and $u_j\in\K_{\psi_1,\psi_2,g_j,f}(\Omega)$ for each $j\in\N,$ it follows that $$\psi_1\le u\le\psi_2$$ $\mu$-a.e.\ in $\Omega.$  Moreover, by \eqref{eqn:u_f} it follows that for each $j\in\N,$ 
$$g_j\le Tu_j\le Tu\le Tu_f=f$$ 
$\Ha$-a.e.\ on $\partial\Omega.$  Since $g_j\to f$ pointwise $\Ha$-a.e., we have that $Tu=f$ $\Ha$-a.e.\ on $\partial\Omega,$ and so $u\in\K_{\psi_1,\psi_2,f,f}(\Omega).$  

By \eqref{eqn:u_fEnergy}, we have that $u$ is a strong solution to the $\K_{\psi_1,\psi_2,f,f}$-obstacle problem, and so by \eqref{eqn:u_f} and minimality of $u_f$, it follows that $u=u_f$ $\mu$-a.e.\ in $\Omega.$ \\

(c)  For each $k\in\N,$ let $u_{f,k}$ be the minimal strong solution to the $\K_{\psi_1,\psi_2,f,h_k}$-obstacle problem, which exists by Theorem~\ref{thm:ExistStrongSoln}.  By Theorem~\ref{cor:DObsCompThm}, it follows that 
$$u_{j,k}\le u_{j+1,k}\le\lim_{j\to\infty}u_{j,k}\le u_{f,k}$$ $\mu$-a.e.\ in $\Omega.$  Let $u_k:=\lim_{j\to\infty}u_{j,k}.$  As in the proof of part (a), we have that $u_{j,k}\to u_k$ in $L^1(\Omega)$ as $j\to\infty,$ and $\psi_1\le u_k\le\psi_2$ $\mu$-a.e.\ in $\Omega$.   Likewise, we have that $$u_{f,k}\in\K_{\psi_1,\psi_2,f,h_k}(\Omega)\subset\K_{\psi_1,\psi_2,g_j,h_k}(\Omega)$$ for each $j\in\N,$ and so from lower-semicontinuity, it follows that 
$$\|Du_k\|(\Omega)\le\liminf_{j\to\infty}\|Du_{j,k}\|(\Omega)\le\|Du_{f,k}\|(\Omega)<\infty.$$  Again by monotonicity, we have that for all $j\in\N,$
$$g_j\le Tu_{j,k}\le Tu_k\le Tu_{f,k}\le h_k$$ $\Ha$-a.e.\ on $\partial\Omega.$  Since $g_j\to f$ $\Ha$-a.e.\, it follows that $f\le Tu_k\le h_k$ $\Ha$-a.e.\ on $\partial\Omega.$  Thus, $u_k\in\K_{\psi_1,\psi_2,f,h_k}(\Omega),$ and it follows that $u_k=u_{f,k}$ $\mu$-a.e.\  That is, $u_k$ is the minimal strong solution to the $\K_{\psi_1,\psi_2,f,h_k}$-obstacle problem. 

Now, letting $u_f$ be the minimal strong solution to the $\K_{\psi_1,\psi_2,f,f}$-obstacle problem, it follows from Theorem~\ref{cor:DObsCompThm} that 
\begin{equation}\label{eqn:u_fLowerBound}u_f\le u=\lim_{k\to\infty}u_k\le u_{k+1}\le u_k\end{equation} 
$\mu$-a.e.\ in $\Omega.$  Again, we have that $u_k\to u$ in $L^1(\Omega),$ and $\psi_1\le u\le\psi_2$ $\mu$-a.e.\ in $\Omega.$  By lower-semicontinuity and the fact that $u_f\in\K_{\psi_1,\psi_2,f,f}(\Omega)\subset\K_{\psi_1,\psi_2,f,h_k}(\Omega)$,
we have that 
$$\|Du\|(\Omega)\le\liminf_{k\to\infty}\|Du_k\|(\Omega)\le\|Du_f\|(\Omega)<\infty,$$ and for all $k\in\N$, we have that 
$$f=Tu_f\le Tu\le Tu_k\le h_k$$ $\Ha$-a.e.\ on $\partial\Omega.$  Since $h_k\to f$ $\Ha$-a.e.\ as $k\to\infty,$ it follows that $Tu=f$ $\Ha$-a.e.\ in $\partial\Omega,$ and so $u\in\K_{\psi_1,\psi_2,f,f}(\Omega).$  Thus $u$ is a strong solution to the $\K_{\psi_1,\psi_2,f,f}$-obstacle problem.  However in this case, \eqref{eqn:u_fLowerBound} does not allow us to determine that $u$ is a minimal strong solution.           
\end{proof}

\vskip.5cm
\noindent Department of Mathematical Sciences, University of Cincinnati, P.O.~Box 210025, Cincinnati, OH~45221-0025, U.S.A.\\
\noindent E-mail: {\tt klinejp@mail.uc.edu}

\end{document}